\documentclass[a4paper,11pt]{article}
\usepackage{amsmath,amsthm,amssymb}
\usepackage[mathscr]{eucal}
\usepackage[bookmarks=false]{hyperref}

\numberwithin{equation}{section}

{\theoremstyle{definition}\newtheorem{definition}{Definition}[section]

\newtheorem{remark}[definition]{Remark}
}

\newtheorem{proposition}[definition]{Proposition}
\newtheorem{lemma}[definition]{Lemma}
\newtheorem{theorem}[definition]{Theorem}

\newenvironment{somop}{\begin{list}{$\bullet$}{\settowidth{\labelwidth}{$\bullet$}\settowidth{\leftmargin}{$\bullet$}\setlength{\rightmargin}{0cm}\setlength{\labelsep}{0.3cm}\addtolength{\leftmargin}{\labelsep}}}{\end{list}}

\newcommand{\M}{\operatorname{M}}
\newcommand{\C}{\mathbb{C}}

\newcommand{\F}{\mathbb{F}}
\newcommand{\cR}{\mathcal{R}}

\newcommand{\actson}{\curvearrowright}
\newcommand{\SL}{\operatorname{SL}}
\newcommand{\rL}{\mathord{\text{\rm L}}}

\newcommand{\N}{\mathbb{N}}

\newcommand{\Z}{\mathbb{Z}}
\newcommand{\cF}{\mathcal{F}}

\newcommand{\cV}{\mathcal{V}}
\newcommand{\id}{\mathord{\operatorname{id}}}
\newcommand{\si}{\sigma}
\newcommand{\cE}{\mathcal{E}}
\newcommand{\recht}{\rightarrow}
\newcommand{\cU}{\mathcal{U}}
\newcommand{\vphi}{\varphi}
\newcommand{\cW}{\mathcal{W}}
\newcommand{\R}{\mathbb{R}}
\newcommand{\al}{\alpha}
\newcommand{\eps}{\varepsilon}

\newcommand{\Tr}{\operatorname{Tr}}
\newcommand{\ovt}{\mathbin{\overline{\otimes}}}
\newcommand{\B}{\operatorname{B}}
\newcommand{\om}{\omega}
\newcommand{\cP}{\mathcal{P}}
\newcommand{\cZ}{\mathcal{Z}}

\newcommand{\cK}{\mathcal{K}}

\newcommand{\cH}{\mathcal{H}}

\newcommand{\ot}{\otimes}

\newcommand{\cG}{\mathcal{G}}
\newcommand{\cM}{\mathcal{M}}

\newcommand{\be}{\beta}

\newcommand{\Mtil}{\widetilde{M}}

\newcommand{\D}{\operatorname{D}}

\newcommand{\cN}{\mathcal{N}}

\newcommand{\cS}{\mathcal{S}}

\newcommand{\cO}{\mathcal{O}}
\newcommand{\cC}{\mathcal{C}}
\newcommand{\cD}{\mathcal{D}}
\newcommand{\TC}{\mathcal{TC}}
\newcommand{\SO}{\operatorname{SO}}
\newcommand{\SU}{\operatorname{SU}}
\newcommand{\ball}{\operatorname{Ball}}
\newcommand{\HNN}{\operatorname{HNN}}
\newcommand{\ox}{\overline{x}}
\newcommand{\oy}{\overline{y}}
\newcommand{\oz}{\overline{z}}
\newcommand{\cT}{\mathcal{T}}

\newcommand{\bim}[3]{\mathord{\raisebox{-0.4ex}[0ex][0ex]{\scriptsize $#1$}{#2}\hspace{-0.2ex}\raisebox{-0.4ex}[0ex][0ex]{\scriptsize $#3$}}}
\newcommand{\almost}{\subset_{\text{\rm approx}}}
\newcommand{\two}{$\|\,\cdot\,\|_2$-norm\ }

\renewcommand{\Re}{\operatorname{Re}}

\begin{document}

\begin{center}
{\LARGE\bf\boldmath One-cohomology and the uniqueness of the group \vspace{0.5ex}\\ measure space decomposition of a II$_1$ factor}

\bigskip

{\sc by Stefaan Vaes\footnote{Partially
    supported by ERC Starting Grant VNALG-200749, Research
    Programme G.0639.11 of the Research Foundation --
    Flanders (FWO) and K.U.Leuven BOF research grant OT/08/032.}\footnote{Department of Mathematics;
    University of Leuven; Celestijnenlaan 200B; B--3001 Leuven (Belgium).
    \\ E-mail: stefaan.vaes@wis.kuleuven.be}}
\end{center}

\begin{abstract}\noindent
We provide a unified and self-contained treatment of several of the recent uniqueness theorems for the group measure space decomposition of a II$_1$ factor. We single out a large class of groups $\Gamma$, characterized by a one-cohomology property, and prove that for every free ergodic probability measure preserving action of $\Gamma$ the associated II$_1$ factor has a unique group measure space Cartan subalgebra up to unitary conjugacy. Our methods follow closely a recent article of Chifan-Peterson, but we replace the usage of Peterson's unbounded derivations by Thomas Sinclair's dilation into a malleable deformation by a one-parameter group of automorphisms.
\end{abstract}

\section{Introduction and main results}

A fundamental problem in the theory of von Neumann algebras is the classification of group measure space II$_1$ factors $\rL^\infty(X) \rtimes \Gamma$ in terms of the initial free ergodic probability measure preserving (pmp) action $\Gamma \actson (X,\mu)$. This problem breaks up in two very different parts. Given an isomorphism $\rL^\infty(X) \rtimes \Gamma \cong \rL^\infty(Y) \rtimes \Lambda$ one first studies whether $\rL^\infty(X)$ and $\rL^\infty(Y)$ are unitarily conjugate. If so this implies that the orbit equivalence relations $\cR(\Gamma \actson X)$ and $\cR(\Lambda \actson Y)$ are isomorphic, leading to the second problem of classifying group actions up to orbit equivalence. This paper deals with the first of these two problems: the uniqueness up to unitary conjugacy of the Cartan subalgebra $\rL^\infty(X) \subset \rL^\infty(X) \rtimes \Gamma$.

A Cartan subalgebra $A$ of a II$_1$ factor $M$ is a maximal abelian von Neumann subalgebra for which the group of unitaries normalizing $A$, i.e.\ $\{u \in \cU(M) \mid uAu^* = A\}$, generates the whole of $M$. If $\Gamma \actson (X,\mu)$ is a free ergodic pmp action, then $\rL^\infty(X) \subset \rL^\infty(X) \rtimes \Gamma$ is a Cartan subalgebra. Not all Cartan subalgebras in a II$_1$ factor can be realized in this way. If they can, we call them group measure space Cartan subalgebras.

By \cite{CFW81} the hyperfinite II$_1$ factor $R$ has a unique Cartan subalgebra up to conjugacy by an automorphism of $R$. Until recently no other uniqueness theorems for Cartan subalgebras were known. A first breakthrough was realized by Ozawa and Popa in \cite{OP07} who proved that the II$_1$ factors $M = \rL^\infty(X) \rtimes \Gamma$ coming from \emph{profinite} free ergodic pmp actions of a direct product of free groups $\Gamma = \F_{n_1} \times \cdots \times \F_{n_k}$, have a unique Cartan subalgebra up to unitary conjugacy. In a second article \cite{OP08} Ozawa and Popa establish the same result for all profinite actions of groups $\Gamma$ satisfying the complete metric approximation property and a strong form of the Haagerup property. This includes lattices in direct products of $\SO(n,1)$ and $\SU(n,1)$.  Peterson then showed in \cite{Pe09} that $M = \rL^\infty(X) \rtimes \Gamma$ has a unique \emph{group measure space} Cartan subalgebra whenever $\Gamma \actson (X,\mu)$ is a profinite action of a free product $\Gamma_1 * \Gamma_2$ where $\Gamma_1$ does not have the Haagerup property and where $\Gamma_2 \neq \{e\}$.

In the joint article \cite{PV09} with Sorin Popa, we introduced a family of amalgamated free product groups $\Gamma = \Gamma_1 *_\Sigma \Gamma_2$ such that for \emph{every} free ergodic pmp action $\Gamma \actson (X,\mu)$, the II$_1$ factor $M:=\rL^\infty(X) \rtimes \Gamma$ has a unique group measure space Cartan subalgebra up to unitary conjugacy.
The family of groups covered by \cite{PV09} consists of the amalgamated free products $\Gamma = \Gamma_1 *_\Sigma \Gamma_2$ such that $\Sigma$ is amenable and weakly malnormal\footnote{By definition we call $\Sigma < \Gamma$ weakly malnormal if there exist $g_1,\ldots,g_n \in \Gamma$ such that $\bigcap_{k=1}^n g_k \Sigma g_k^{-1}$ is finite.} in $\Gamma$ and such that $\Gamma$ either admits a nonamenable subgroup with the relative property (T) or admits two commuting nonamenable subgroups. In particular the result holds for all free products $\Gamma_1 * \Gamma_2$ where $\Gamma_1$ is an infinite property (T) group and $\Gamma_2 \neq \{e\}$.

Combining the uniqueness of the group measure space Cartan subalgebra with existing orbit equivalence superrigidity results from \cite{Po05,Ki09}, we proved in \cite{PV09} several W$^*$-superrigidity theorems: a free ergodic pmp action $\Gamma \actson (X,\mu)$ is called W$^*$-superrigid if the II$_1$ factor $M=\rL^\infty(X) \rtimes \Gamma$ entirely remembers the group action. This means that whenever $M = \rL^\infty(Y) \rtimes \Lambda$, the groups $\Gamma$ and $\Lambda$ are isomorphic and their actions are conjugate.

In \cite{FV10} we proved a similar unique group measure space decomposition theorem for all free ergodic pmp actions of certain HNN extensions $\Gamma = \HNN(H,\Sigma,\theta)$, where $\Sigma$ is amenable and $\Gamma$ satisfies a rigidity assumption as above. In \cite{HPV10} we realized that also certain amalgamated free products $\Gamma = \Gamma_1 *_\Sigma \Gamma_2$ over nonamenable groups $\Sigma$ could be covered by the methods of \cite{PV09}. Combined with a theorem of Kida \cite{Ki09} we deduced in \cite{HPV10} that \emph{every} free ergodic pmp action of the group $\Gamma = \SL(3,\Z)*_\Sigma \SL(3,\Z)$ is W$^*$-superrigid, when $\Sigma < \SL(3,\Z)$ denotes the subgroup of matrices $g$ with $g_{31}=g_{32}=0$.

In the very recent article \cite{CP10}, Chifan and Peterson proposed a more conceptual framework to prove the uniqueness of the group measure space decomposition. Their theorem covers all groups $\Gamma$ that admit a nonamenable subgroup with the relative property (T) and that admit an unbounded $1$-cocycle into a mixing orthogonal representation. The latter means that there exists an orthogonal representation $\pi : \Gamma \recht \cO(H_\R)$ and an unbounded map $b : \Gamma \recht H_\R$ satisfying $b(gh) = b(g) + \pi(g) b(h)$ for all $g,h \in \Gamma$, such that for all $\xi,\eta \in H_\R$ we have $\langle \pi(g)\xi,\eta\rangle \recht 0$ as $g \recht \infty$. Typical examples arise as free products $\Gamma = \Gamma_1 * \Gamma_2$ containing a nonamenable subgroup with the relative property (T). But Chifan and Peterson cover as well direct products $\Gamma \times \Gamma'$ of such groups. This fits perfectly with Monod-Shalom's  orbit equivalence rigidity theorems for direct product groups (see \cite{MS02}) and leads to new W$^*$ strong rigidity results.

Amalgamated free products and HNN extensions also admit unbounded $1$-cocycles into orthogonal representations $\pi$ through the action on their Bass-Serre tree, but these representations $\pi$ are only mixing relative to the amalgam $\Sigma$. In this paper we generalize Chifan-Peterson's result to groups that admit an unbounded $1$-cocycle into an orthogonal representation that is mixing relative to a family of amenable subgroups. As such we obtain a unified treatment for all the uniqueness theorems of the group measure space decomposition in \cite{PV09,FV10,HPV10,CP10}.

Although our methods are very close to those in \cite{CP10}, we do not use Peterson's technique of unbounded derivations (\cite{Pe06} and \cite[Section 4]{OP08}), but rather their dilation into a malleable deformation in the sense of Popa, as proposed by Thomas Sinclair \cite{Si10}. In this way our approach becomes more elementary and we can more directly apply the methods from Popa's deformation/rigidity theory (see \cite{Po06a,Va10} for an overview).

\subsection*{Statements of the main results}

We say that a countable group $\Sigma$ is anti-(T) if there exists a chain of subgroups $\{e\} = \Sigma_0 < \Sigma_1 < \cdots < \Sigma_n = \Sigma$ such that for all $i=1,\ldots,n$ the subgroup $\Sigma_{i-1}$ is normal in $\Sigma_i$ and the quotient $\Sigma_i / \Sigma_{i-1}$ has the Haagerup property.

We say that an orthogonal representation $\pi : \Gamma \recht \cO(H_\R)$ is mixing relative to a family $\cS$ of subgroups of $\Gamma$ if the following holds: for all $\xi,\eta \in H_\R$ and $\eps > 0$ there exist a finite number of $g_i,h_i \in \Gamma$ and $\Sigma_i \in \cS$ such that $|\langle \pi(g)\xi,\eta \rangle| < \eps$ for all $g \in \Gamma - \bigcup_{i=1}^n g_i \Sigma_i h_i$.

\begin{definition}\label{def.classgroups}
We consider three classes of countable groups $\Gamma$ that admit an unbounded $1$-cocycle $b : \Gamma \recht H_\R$ into an orthogonal representation $\pi : \Gamma \recht \cO(H_\R)$ that is mixing relative to a family $\cS$ of subgroups of $\Gamma$, with $b$ being bounded on every $\Sigma \in \cS$. These three classes correspond to imposing a rigidity on $\Gamma$ versus a softness on the groups in $\cS$.

{\bf Class $\cC$.} $\Gamma$ has a nonamenable subgroup with the relative property (T) and the groups in $\cS$ are amenable.

{\bf Class $\cD$.} $\Gamma$ has an infinite subgroup with the plain property (T) and the groups in $\cS$ are anti-(T).

{\bf Class $\cE$.} $\Gamma$ has two commuting nonamenable subgroups, the groups in $\cS$ are amenable and $\pi$ is weakly contained in the regular representation.

We also consider the classes $\cC_2$, respectively $\cD_2$, consisting of direct products $\Gamma_1 \times \Gamma_2$ where both $\Gamma_i \in \cC$, respectively both $\Gamma_i \in \cD$. The groups $\Gamma_i$ come with a family $\cS_i$ of subgroups and we consider the family $\cS$ of subgroups of $\Gamma$ of the form $\Sigma_1 \times \Sigma_2$, $\Sigma_i \in \cS_i$.
\end{definition}

The following is our main theorem. Given a group $\Gamma$ in any of the classes introduced above, we \lq locate\rq\ any possible group measure space Cartan subalgebra $B$ of any crossed product $A \rtimes \Gamma$ and prove that it must have an intertwining bimodule into $A \rtimes \Sigma$ for some $\Sigma \in \cS$. We refer to Theorem \ref{thm.intertwining} below for the definition of Popa's intertwining bimodules and the corresponding notation $\prec$. If the groups $\Sigma \in \cS$ are moreover finite or sufficiently nonnormal, it follows that $\rL^\infty(X) \rtimes \Gamma$ has a unique group measure space Cartan subalgebra up to unitary conjugacy for all free ergodic pmp actions $\Gamma \actson (X,\mu)$, see Theorem \ref{thm.uniquecartan}. If moreover $\Gamma \actson (X,\mu)$ is orbit equivalence superrigid, the action $\Gamma \actson (X,\mu)$ follows W*-superrigid in the sense that the II$_1$ factor $\rL^\infty(X) \rtimes \Gamma$ entirely remembers the group action that it was constructed from.

\begin{theorem}\label{thm.main}
Let $\Gamma$ be a group in $\cC \cup \cD \cup \cE \cup \cC_2 \cup \cD_2$ together with its family $\cS$ of subgroups as in Definition \ref{def.classgroups}.

Let $M$ be a II$_1$ factor of the form $M = A \rtimes \Gamma$ where $A$ is of type I. Let $p \in M$ be a projection and assume that $pMp = B \rtimes \Lambda$ is another crossed product decomposition with $B$ being of type I. Then there exists $\Sigma \in \cS$ such that $B \prec A \rtimes \Sigma$.
\end{theorem}

Theorem \ref{thm.main} is sufficiently general to cover all the uniqueness theorems of group measure space Cartan subalgebras from \cite{PV09,FV10,HPV10,CP10}. The precise formulation goes as follows.

We say that $\Sigma \subset \Gamma$ is a weakly malnormal subgroup if there exist $g_1,\ldots,g_n \in \Gamma$ such that $\bigcap_{k=1}^n g_k \Sigma g_k^{-1}$ is finite.
We say that $\Sigma \subset \Gamma$ is relatively malnormal if there exists a subgroup $\Lambda < \Gamma$ of infinite index such that $g \Sigma g^{-1} \cap \Sigma$ is finite for all $g \in \Gamma - \Lambda$.

We consider amalgamated free products $\Gamma_1 *_\Sigma \Gamma_2$ and call them nontrivial when $\Gamma_1 \neq \Sigma \neq \Gamma_2$. We also consider HNN extensions $\HNN(H,\Sigma,\theta)$ w.r.t.\ a subgroup $\Sigma < H$ and an injective group homomorphism $\theta : \Sigma \recht H$, generated by a copy of $H$ and an extra generator $t$ satisfying $t \sigma t^{-1} = \theta(\sigma)$ for all $\sigma \in \Sigma$.

\begin{theorem}\label{thm.uniquecartan}
For all of the following groups $\Gamma$ and arbitrary free ergodic pmp actions $\Gamma \actson (X,\mu)$, the II$_1$ factor $M = \rL^\infty(X) \rtimes \Gamma$ has a unique group measure space Cartan subalgebra up to unitary conjugacy. More generally, if $p(\M_n(\C) \ot M)p = \rL^\infty(Y) \rtimes \Lambda$ is an arbitrary group measure space decomposition, there exists $u \in \M_n(\C) \ot M$ such that $p=u^* u$, such that $q:=u u^*$ belongs to $\D_n(\C) \ot \rL^\infty(X)$ and such that $u \rL^\infty(Y) u^* = (\D_n(\C) \ot \rL^\infty(X))q$. Here $\D_n(\C) \subset \M_n(\C)$ denotes the subalgebra of diagonal matrices.
\begin{enumerate}
\item Any of the groups $\Gamma \in \cC \cup \cD \cup \cE \cup \cC_2 \cup \cD_2$ such that the family $\cS$ consists of relatively malnormal subgroups of $\Gamma$.
\item \cite[Corollary 5.3]{CP10}. Any of the groups $\Gamma \in \cC \cup \cD \cup \cE \cup \cC_2 \cup \cD_2$ such that the family $\cS$ is reduced to $\{\{e\}\}$, i.e.\ the case where the orthogonal representations are mixing.
\item \cite[Theorem 5.2]{PV09}. $\Gamma$ is a nontrivial amalgamated free product $\Gamma_1 *_\Sigma \Gamma_2$ where $\Sigma$ is amenable and weakly malnormal in $\Gamma$ and where $\Gamma$ admits a nonamenable subgroup with the relative property (T) or admits two commuting nonamenable subgroups.
\item \cite[Theorem 5]{HPV10}. $\Gamma$ is a nontrivial amalgamated free product $\Gamma_1 *_\Sigma \Gamma_2$ where $\Sigma$ is anti-(T) and weakly malnormal in $\Gamma$ and where $\Gamma$ admits an infinite subgroup with property (T).
\item \cite[Theorem 4.1]{FV10}. $\Gamma$ is an HNN extension $\HNN(H,\Sigma,\theta)$ such that $\Sigma$ is amenable and weakly malnormal in $\Gamma$ and such that $\Gamma$ admits a nonamenable subgroup with the relative property (T) or admits two commuting nonamenable subgroups.
\end{enumerate}
\end{theorem}

Ideally Theorem \ref{thm.uniquecartan} could hold for all groups $\Gamma$ in $\cC \cup \cD \cup \cE \cup \cC_2 \cup \cD_2$ for which $\cS$ consists of weakly malnormal subgroups of $\Gamma$, but we were unable to prove such a statement.

As in \cite[Corollary 4.4]{OP07} and \cite[Theorem 1.4]{PV09} we also get plenty of II$_1$ factors that have no group measure space Cartan subalgebra.

\begin{theorem}\label{thm.no-group-measure-space}
Let $\Gamma$ be any of the groups in Theorem \ref{thm.uniquecartan} and assume moreover that $\Gamma$ has infinite conjugacy classes (icc). Let $\Gamma \actson (X,\mu)$ be an ergodic pmp action that is not essentially free. Denote $M = \rL^\infty(X) \rtimes \Gamma$ and observe that $M$ is a II$_1$ factor. This includes the case where $X$ is one point and $M = \rL \Gamma$. The II$_1$ factors $M^t$, $t > 0$, have no group measure space decomposition.
\end{theorem}

In \cite{Ki10} Kida proves measure equivalence rigidity theorems for amalgamated free product groups $\Gamma$. Such results are complementary to uniqueness theorems for the group measure space Cartan subalgebra since the latter reduce a W*-equivalence to an orbit equivalence to which the measure equivalence rigidity theorems can be applied. It is therefore interesting to notice that Kida does not use the $1$-cocycle that goes with the action of $\Gamma$ on the Bass-Serre tree $\cT$, but rather the very much related map that associates to three distinct points of $\ox,\oy,\oz \in \partial \cT$ the unique vertex of the tree where the three geodesics $\ox\,\oy$, $\oy\,\oz$ and $\oz\,\ox$ meet.

We have made the choice to write an essentially self-contained article with rather elementary proofs. This means that we provide detailed arguments for a number of lemmas that are well known to some experts. It also gives us the occasion to reprove a number of results, e.g.\ Peterson's theorem on the solidity of certain group von Neumann algebras, see Theorem \ref{thm.solid}.

\subsection*{Structure of the proofs}

Let $M$ be a II$_1$ factor of the form $M = \rL^\infty(X) \rtimes \Gamma$, where $\Gamma$ is a countable group and $b : \Gamma \recht H_\R$ is an unbounded $1$-cocycle into an orthogonal representation $\pi : \Gamma \recht \cO(H_\R)$. We assume that $\pi$ is mixing relative to a family $\cS$ of subgroups of $\Gamma$ and that $b$ is bounded on every $\Sigma \in \cS$.
Following \cite[Definition 15.1.1]{BO08} we call a subset $\cF \subset \Gamma$ small relative to $\cS$ when $\cF$ can be written as a finite union of subsets of the form $g \Sigma h$, $g,h \in \Gamma$, $\Sigma \in \cS$.

Associated with $(b,\pi)$ is the conditionally negative type function $\psi(g) = \|b(g)\|^2$ and the semigroup group $(\vphi_t)_{t > 0}$ of completely positive maps
$$\vphi_t : M \recht M : \vphi_t(a u_g) = \exp(-t \psi(g)) \; a u_g \quad\text{for all}\;\; a \in \rL^\infty(X) , g \in \Gamma \; .$$
Assume that $M = \rL^\infty(Y) \rtimes \Lambda$ is another group measure space decomposition. Also assume that $\Gamma$ satisfies a rigidity assumption, like the presence of a nonamenable subgroup with the (relative) property (T) or the presence of two commuting nonamenable subgroups.

\begin{somop}
\item It is of course impossible to prove that $\Lambda$ automatically inherits similar rigidity properties as $\Gamma$ has. Nevertheless the transfer of rigidity principle from \cite{PV09} allows to prove that there exists a sequence of group elements $s_n \in \Lambda$ such that $\vphi_t \recht \id$ in \two uniformly on $\{v_{s_n} \mid n \in \N\}$. Also the sequence can be taken \lq large enough\rq\ in the sense that $v_{s_n}$ asymptotically leaves all the subspaces spanned by $\{a u_k \mid a \in \rL^\infty(X), k \in \cF\}$ for any fixed subset $\cF \subset \Gamma$ that is small relative to $\cS$. These matters are dealt with in Proposition \ref{prop.transfer} and Lemma \ref{lem.no-intertwine}. The \lq larger\rq\ the subgroups in $\cS$ are allowed to be, the stricter the rigidity assumption on $\Gamma$ must be.

\item The unitaries $v_{s_n}$ normalize the abelian von Neumann subalgebra $\rL^\infty(Y)$ and the deformation $\vphi_t$ converges to the identity uniformly on the $(v_{s_n})_{n \in \N}$. In Theorem \ref{thm.uniform-abelian} we prove that then $\vphi_t$ must converge to the identity uniformly on the unit ball of $\rL^\infty(Y)$. The roots of this result lie in \cite[Theorem 4.1]{Pe09}. Our theorem and its proof are almost identical to \cite[Theorem 3.2]{CP10}, but we manage to treat as well the case of nonmixing representations.

    It is illustrative to compare Theorem \ref{thm.uniform-abelian} and its proof to the following group theoretic result inspired by \cite{CTV06}. Let $b : \Gamma \recht H_\R$ be a $1$-cocycle into an orthogonal representation $\pi : \Gamma \recht \cO(H_\R)$ that is mixing relative to a family $\cS$ of subgroups of $\Gamma$. Assume that $b$ is bounded on every subgroup $\Sigma \in \cS$. We say that a sequence $(g_n)$ tends to infinity relative to $\cS$ if $(g_n)$ eventually leaves every subset of $\Gamma$ that is small relative to $\cS$. Whenever $H < \Gamma$ is an abelian subgroup that is normalized by a sequence $(g_n)$ tending to infinity relative to $\cS$ and on which $b$ is bounded, then the $1$-cocycle $b$ is bounded on $H$. The proof consists of two cases. Put $\kappa = \sup_n \|b(g_n)\| < \infty$.

    {\bf Case 1.} There is no $h \in H$ such that the sequence $g_n h g_n^{-1}$ tends to infinity relative to $\cS$. We show that $\|b(h)\| \leq 2\kappa$ for all $h \in H$. To prove this statement, fix $h \in H$. Since $g_n h g_n^{-1}$ does not tend to infinity relative to $\cS$, we can pass to a subsequence and find $\gamma, \gamma' \in \Gamma$ and $\Sigma \in \cS$ such that $g_n h g_n^{-1} \in \gamma \Sigma \gamma'$ for all $n$. Since $b$ is bounded on $\Sigma$, there exists a vector $\xi \in H_\R$ such that $b(\sigma) = \pi(\sigma)\xi - \xi$ for all $\sigma \in \Sigma$. We then find $\xi_1,\xi_2 \in H_\R$ such that $b(g) = \pi(g)\xi_1 + \xi_2$ for all $g \in \gamma \Sigma \gamma'$. In particular $b(g_n h g_n^{-1}) = \pi(g_n h g_n^{-1})\xi_1 + \xi_2$ for all $n$. Since $b(g_n h g_n^{-1}) = b(g_n) + \pi(g_n) b(h) - \pi(g_n h g_n^{-1})b(g_n)$ we conclude that
\begin{align*}
\|b(h)\|^2 &= \langle \pi(g_n) b(h) , \pi(g_n) b(h)\rangle = \langle \pi(g_n) b(h), b(g_n hg_n^{-1}) - b(g_n) + \pi(g_nhg_n^{-1}) b(g_n) \rangle \\
& \leq 2 \kappa \|b(h)\| + |\langle \pi(g_n) b(h) , \pi(g_n h g_n^{-1})\xi_1 + \xi_2 \rangle| \\
& \leq 2 \kappa \|b(h)\| + |\langle \pi(g_n) \pi(h)^* b(h), \xi_1\rangle| + |\langle \pi(g_n) b(h), \xi_2 \rangle| \recht 2 \kappa \|b(h)\|
\end{align*}
because $(g_n)$ tends to infinity relative to $\cS$ and $\pi$ is mixing relative to $\cS$. We have shown that $\|b(h)\| \leq 2 \kappa$ for all $h \in H$.

{\bf Case 2.} There exists $h_0 \in H$ such that the sequence $h_n := g_n h_0 g_n^{-1}$ tends to infinity relative to $\cS$. Put $\kappa_1 = 2 \kappa + \|b(h_0)\|$. Note that $\|b(h_n)\| \leq \kappa_1$ for all $n$. We show that $\|b(h)\| \leq 2 \kappa_1$ for all $h \in H$. To prove this statement, fix $h \in H$. Since $H$ is abelian, we have $h = h_n h h_n^{-1}$ and hence
\begin{align*}
\|b(h)\|^2 &= \langle b(h_n h h_n^{-1}), b(h) \rangle = \langle b(h_n) + \pi(h_n) b(h) - \pi(h_n h h_n^{-1}) b(h_n), b(h)\rangle \\
& \leq 2 \kappa_1 \|b(h)\| + |\langle \pi(h_n) b(h),b(h)\rangle| \recht 2 \kappa_1 \|b(h)\|\; .
\end{align*}
We have shown that $\|b(h)\| \leq 2 \kappa_1$ for all $h \in H$.

\item From the previous step we know that $\vphi_t \recht \id$ in \two uniformly on the unit ball of $\rL^\infty(Y)$. We argue by contradiction that there exists a $\Sigma \in \cS$ such that $\rL^\infty(Y) \prec \rL^\infty(X) \rtimes \Sigma$. If the statement is false Lemma \ref{lem.intertwining} provides a sequence of unitaries $b_n \in \rL^\infty(Y)$ such that $b_n$ asymptotically leaves all the subspaces spanned by $\{a u_k \mid a \in \rL^\infty(X), k \in \cF\}$ for any fixed subset $\cF \subset \Gamma$ that is small relative to $\cS$.

    By \cite[Theorem 2.5]{CP10} --for which we provide a self-contained proof as Theorem \ref{thm.uniform-normalizer} below-- it follows that $\vphi_t \recht \id$ in \two uniformly on the unit ball of the normalizer of $\rL^\infty(Y)$, i.e.\ on the unit ball of the whole of $M$. This is absurd because the $1$-cocycle $b$ was assumed to be unbounded.

    Also Theorem \ref{thm.uniform-normalizer} and its proof can be illustrated by a well known group theoretic fact. Let $b : \Gamma \recht H_\R$ be a $1$-cocycle into an orthogonal representation $\pi : \Gamma \recht \cO(H_\R)$ that is mixing relative to a family $\cS$ of subgroups of $\Gamma$. Assume that $H < \Gamma$ is a subgroup of $\Gamma$ on which the cocycle is bounded. Denote by $\cN_\Gamma(H)$ the normalizer of $H$ inside $\Gamma$. If $H$ contains a sequence $(h_n)$ tending to infinity relative to $\cS$ (meaning that $(h_n)$ eventually leaves every subset of $\Gamma$ that is small relative to $\cS$), then $b$ is bounded on $\cN_\Gamma(H)$. The proof of this fact goes as follows.

    Put $\kappa = \sup_{h \in H}\|b(h)\| < \infty$. We show that $\|b(g)\| \leq 2\kappa$ for all $g \in \cN_\Gamma(H)$. To prove this statement fix $g \in \cN_\Gamma(H)$. Write $k_n := g^{-1} h_n^{-1} g$ and note that $k_n \in H$. By construction $g = h_n g k_n$. Therefore
    \begin{align*}
    \|b(g)\|^2 &= \langle b(g),b(g)\rangle = \langle b(h_n g k_n),b(g)\rangle = \langle b(h_n) +\pi(h_n) b(g) + \pi(h_ng) b(k_n),b(g)\rangle \\
    & \leq 2\kappa \|b(g)\| + |\langle \pi(h_n) b(g),b(g)\rangle| \recht 2 \kappa \|b(g)\| \; .
    \end{align*}
    It follows that $\|b(g)\| \leq 2 \kappa$ for all $g \in \cN_\Gamma(H)$.

\item Once we know that $\rL^\infty(Y) \prec \rL^\infty(X) \rtimes \Sigma$ for some $\Sigma \in \cS$, the unitary conjugacy of $\rL^\infty(X)$ and $\rL^\infty(X)$ follows from a combination of the regularity of $\rL^\infty(Y) \subset M$ and a weak malnormality of $\Sigma$ in $\Gamma$, see Lemma \ref{lem.malnormal} and  \cite[Proposition 8]{HPV10}.
\end{somop}

\subsection*{Acknowledgment}

It is my pleasure to thank Claire Anantharaman, Cyril Houdayer, Jesse Peterson, Sorin Popa and Thomas Sinclair for their helpful comments and careful reading of a first draft of this article.

\section{Popa's intertwining-by-bimodules}

We first recall Popa's intertwining-by-bimodules theorem.

\begin{theorem}[{\cite[Theorem 2.1 and Corollary 2.3]{Po03}}]\label{thm.intertwining}
Let $(M,\tau)$ be a von Neumann algebra with a faithful normal tracial state $\tau$. Assume that $p \in M$ is a projection and that $P \subset M$ and $B \subset pMp$ are von Neumann subalgebras with $B$ being generated by a group of unitaries $\cG \subset \cU(B)$. Then the following two statements are equivalent.
\begin{itemize}
\item There exists a nonzero partial isometry $v \in \M_{1,n}(\C) \ot p M$ and a, possibly non-unital, normal $*$-homomorphism $\theta : B \recht \M_n(\C) \ot P$ such that $b v = v \theta(b)$ for all $b \in B$.
\item There is no sequence of unitaries $(b_n)$ in $\cG$ satisfying $\|E_P(x b_n y)\|_2 \recht 0$ for all $x,y \in M$.
\end{itemize}
We write $B \prec P$ if these equivalent conditions hold.
\end{theorem}

Note that when the von Neumann algebra $M$ is non separable, sequences have to be replaced by nets in the formulation of Theorem \ref{thm.intertwining}.

Throughout this section we fix a trace preserving action $\Gamma \actson (N,\tau)$ and denote $M = N \rtimes \Gamma$. We assume that $\cS$ is a family of subgroups of $\Gamma$. Following \cite[Definition 15.1.1]{BO08} we say that a subset $\cF \subset \Gamma$ is small relative to $\cS$ if $\cF$ can be written as a finite union of subsets of the form $g \Sigma h$, $g,h \in \Gamma, \Sigma \in \cS$. Whenever $\cF \subset \Gamma$ we denote by $P_\cF$ the orthogonal projection of $\rL^2(M)$ onto the closed linear span of $\{a u_g \mid a \in N, g \in \cF\}$.

\begin{definition}\label{def.almost}
Whenever $\cV \subset M$ is a norm bounded subset, we write $\cV \almost N \rtimes \cS$ if for every $\eps > 0$ there exists a subset $\cF \subset \Gamma$ that is small relative to $\cS$ such that
$$\|b - P_\cF(b)\|_2 = \|P_{\Gamma-\cF}(b)\|_2 \leq \eps \quad\text{for all}\;\; b \in \cV \; .$$
\end{definition}

\begin{lemma}\label{lem.almost}
Let $\cV \subset M$ be a norm bounded subset satisfying $\cV \almost N \rtimes \cS$. Then for all $x,y \in M$ we have $x \cV y \almost N \rtimes \cS$.

If $(v_i)$ is a bounded net in $M$ satisfying $\|P_\cF(v_i)\|_2 \recht 0$ for every subset $\cF \subset \Gamma$ that is small relative to $\cS$, then $\|P_\cF(x v_i y)\|_2 \recht 0$ for every $x,y \in M$ and every subset $\cF \subset \Gamma$ that is small relative to $\cS$.
\end{lemma}
\begin{proof}
To prove the first statement, by symmetry it suffices to show that $x \cV \almost N \rtimes \cS$. We may assume that $x \in (M)_1$ and $\cV \subset (M)_1$. Choose $\eps > 0$. Take a finite subset $\cF_0 \subset \Gamma$ and elements $(a_g)_{g \in \cF_0}$ of $N$ such that $x_0 := \sum_{g \in \cF_0} a_g u_g$ satisfies $\|x-x_0\|_2 \leq \eps/2$. Put $\kappa = \max\{\|a_g\| \mid g \in \cF_0\}$. Take a subset $\cF \subset \Gamma$ that is small relative to $\cS$ and such that $\|P_{\Gamma-\cF}(b)\|_2 \leq \eps/(2|\cF_0| \kappa)$ for all $b \in \cV$. Define the subset $\cF_1 := \cF_0 \cF$ and note that $\cF_1$ is small relative to $\cS$. We claim that $\|P_{\Gamma-\cF_1}(x b)\|_2 \leq \eps$ for all $b \in \cV$. To prove this claim, fix $b \in \cV$. Since $\|b\| \leq 1$, we have $\|xb-x_0 b\|_2 \leq \|x-x_0\|_2 \leq \eps/2$. So it suffices to prove that $\|P_{\Gamma-\cF_1}(x_0 b)\|_2 \leq \eps/2$. But,
$$P_{\Gamma-\cF_1}(x_0 b) = \sum_{g \in \cF_0} a_g P_{\Gamma - \cF_1}(u_g b)$$
so that
$$\|P_{\Gamma-\cF_1}(x_0 b)\|_2 \leq \sum_{g \in \cF_0} \|a_g\| \; \|P_{\Gamma-\cF_1}(u_g b)\|_2 \leq \kappa \sum_{g \in \cF_0}  \|P_{\Gamma-\cF}(b)\|_2 \leq \frac{\eps}{2}\; .$$
This proves the claim and hence also the first statement in the lemma.

To prove the second statement, we can approximate $x$ and $y$ by linear combinations of $a u_g$, $a \in N$, $g \in \Gamma$. Using the Kaplansky density theorem we may assume that $x = u_g$ and $y = u_h$. But then $P_\cF(u_g v_i u_h) = P_{g^{-1} \cF h^{-1}}(v_i)$ and we are done.
\end{proof}

The following lemma is essentially contained in \cite{Po03} and \cite[Remark 3.3]{Va07}. To keep this paper self-contained we provide a short proof. The only reason that nets appear, is because we do not want to make the restriction that the family $\cS$ is countable.

\begin{lemma}\label{lem.intertwining}
Let $p \in M$ be a projection and $B \subset pMp$ a von Neumann subalgebra generated by a group of unitaries $\cG \subset \cU(B)$. Then the
following two statements are equivalent.
\begin{itemize}
\item For every $\Sigma \in \cS$ we have $B \not\prec N \rtimes \Sigma$.
\item There exists a net of unitaries $(w_i)$ in $\cG$ such that $\|P_\cF(w_i)\|_2 \recht 0$ for every subset $\cF \subset \Gamma$ that is small relative to $\cS$.
\end{itemize}
\end{lemma}
\begin{proof}
If $\Sigma \in \cS$ and $B \prec N \rtimes \Sigma$, Theorem \ref{thm.intertwining} yields a nonzero partial isometry $v \in \M_{1,n}(\C) \ot pM$ and a normal $*$-homomorphism $\theta : B \recht \M_n(\C) \ot (N \rtimes \Sigma)$ such that $b v = v \theta(b)$ for all $b \in B$. If $(w_i)$ would be a net as in the second statement, Lemma \ref{lem.almost} implies that $\|(1 \ot P_\cF)(v^* w_i v)\|_2 \recht 0$ for every subset $\cF \subset \Gamma$ that is small relative to $\cS$. This holds in particular for $\cF = \Sigma$ so that
$$\|(\id \ot E_{N \rtimes \Sigma})(v^* v)\|_2 = \|\theta(w_i)\; (\id \ot E_{N \rtimes \Sigma})(v^* v)\|_2 = \|(\id \ot E_{N \rtimes \Sigma})(v^* w_i v)\|_2 \recht 0 \; .$$
We arrive at the contradiction that $v = 0$.

Conversely assume that for every $\Sigma \in \cS$ we have $B \not\prec N \rtimes \Sigma$. Let $\cF \subset \Gamma$ be a subset that is small relative to $\cS$ and let $\eps > 0$. We have to prove the existence of $w \in \cG$ such that $\|P_\cF(w)\|_2 < \eps$. Write $\cF = \bigcup_{k=1}^n g_k \Sigma_k h_k$ with $\Sigma_k \in \cS$ and $g_k,h_k \in \Gamma$. Consider in $\M_n(\C) \ot M$ the diagonal subalgebra $P:= \bigoplus_k N \rtimes \Sigma_k$. Since for every $k =1,\ldots,n$ we have $B \not\prec N \rtimes \Sigma_k$, the first criterion of Theorem \ref{thm.intertwining} implies that $B \not\prec P$. Then the second criterion in Theorem \ref{thm.intertwining} provides a sequence of unitaries $w_i \in \cG$ such that
$$\|E_{N \rtimes \Sigma_k}(x w_i y)\|_2 \recht 0 \quad\text{for all}\;\; x,y \in M, k=1,\ldots,n \; .$$
Applying this to $x = u_{g_k}^*$ and $y = u_{h_k}^*$, this means that $\|P_{g_k \Sigma_k h_k}(w_i)\|_2 \recht 0$. Hence $\|P_\cF(w_i)\|_2 \recht 0$ and it suffices to take $w = w_i$ for $i$ sufficiently large.
\end{proof}

The following lemma clarifies the relation between the approximate containment $\almost$ and the intertwining relation $\prec$.

\begin{lemma}\label{lem.approx-vs-prec}
Let $p \in M$ be a projection and $B \subset pMp$ a von Neumann subalgebra. The following two statements are equivalent.
\begin{enumerate}
\item There exists a $\Sigma \in \cS$ such that $B \prec N \rtimes \Sigma$.\label{stat.1}
\item There exists a nonzero projection $q \in B' \cap pMp$ such that $(Bq)_1 \almost N \rtimes \cS$.\label{stat.2}
\end{enumerate}
Also the following two statements are equivalent.
\begin{enumerate}\renewcommand{\theenumi}{\alph{enumi}}
\item For every nonzero projection $q \in B' \cap pMp$ there exists a $\Sigma \in \cS$ such that $Bq \prec N \rtimes \Sigma$.\label{stat.a}
\item We have $(B)_1 \almost N \rtimes \cS$.\label{stat.b}
\end{enumerate}
\end{lemma}
\begin{proof}
\ref{stat.1} $\Rightarrow$ \ref{stat.2}. Take $\Sigma \in \cS$ such that $B \prec N \rtimes \Sigma$. Theorem \ref{thm.intertwining} provides a nonzero partial isometry $v \in \M_{1,n}(\C) \ot p M$ such that $q := v v^*$ belongs to $B' \cap pMp$ and satisfies $B q \subset v (\M_n(\C) \ot (N \rtimes \Sigma)) v^*$. It follows from Lemma \ref{lem.almost} that $(Bq)_1 \almost N \rtimes \cS$.

$\neg$ \ref{stat.1} $\Rightarrow$ $\neg$ \ref{stat.2}. We assume that for all $\Sigma \in \cS$ we have $B \not\prec N \rtimes \Sigma$. Then Lemma \ref{lem.intertwining} provides a net of unitaries $(w_i)$ in $\cU(B)$ such that $\|P_\cF(w_i)\|_2 \recht 0$ for every subset $\cF \subset \Gamma$ that is small relative to $\cS$. Take a nonzero projection $q \in B' \cap pMp$. We conclude from Lemma \ref{lem.almost} that $\|P_\cF(w_i q)\|_2 \recht 0$ for every subset $\cF \subset \Gamma$ that is small relative to $\cS$. Since $w_i q \in (Bq)_1$ and $\|w_i q\|_2 = \|q\|_2 > 0$ for all $i$, we see that $(Bq)_1$ is not approximately contained in $N \rtimes \cS$.

\ref{stat.a} $\Rightarrow$ \ref{stat.b}. By the assumption in statement~\ref{stat.a} and using the already proven implication \ref{stat.1}~$\Rightarrow$~\ref{stat.2}, we can take an orthogonal family of projections $q_i \in B' \cap pMp$ such that $\sum_i q_i = p$ and $(B q_i)_1 \almost N \rtimes \cS$ for all $i$. It follows that $(B)_1 \almost N \rtimes \cS$.

\ref{stat.b} $\Rightarrow$ \ref{stat.a}. Assume that $(B)_1 \almost N \rtimes \cS$ and let $q \in B' \cap pMp$ be a nonzero projection. We conclude from Lemma \ref{lem.almost} that $(Bq)_1 \almost N \rtimes \cS$. The already proven implication \ref{stat.2}~$\Rightarrow$~\ref{stat.1} implies that there exists a $\Sigma \in \cS$ such that $Bq \prec N \rtimes \Sigma$.
\end{proof}

Whenever $p \in M$ is a projection and $B \subset pMp$ is a von Neumann subalgebra we find as follows a unique projection $q \in B' \cap pMp$ such that $(B q)_1 \almost N \rtimes \cS$ and such that $B(p-q) \not\prec N \rtimes \Sigma$ for all $\Sigma \in \cS$.

\begin{proposition}\label{prop.maxintertwine}
Let $p \in M$ be a projection and $B \subset pMp$ a von Neumann subalgebra generated by a group of unitaries $\cG \subset \cU(B)$. Denote by $Q$ the normalizer of $B$ inside $pMp$.
The set of projections
$$\cP := \{q_0 \mid q_0 \;\;\text{is a projection in}\;\; B' \cap pMp \;\;\text{and}\;\; (Bq_0)_1 \almost N \rtimes \cS \}$$
attains its maximum in a unique projection $q \in \cP$. This projection $q$ belongs to $\cZ(Q)$.

Moreover there exists a net of unitaries $(w_i)$ in $\cG$ such that $\|P_\cF(w_i(p-q))\|_2 \recht 0$ for every subset $\cF \subset \Gamma$ that is small relative to $\cS$.
\end{proposition}

\begin{proof}
Let $q_k \in B' \cap pMp$ be a maximal orthonormal sequence of nonzero projections in $\cP$. Define $q = \sum_k q_k$. It is easy to prove that $q \in \cP$. We prove that $q$ is the maximum of $\cP$.
Assume that $q' \in \cP$ and that $q' \not\leq q$. It follows that $T:= (p-q) q' (p-q)$ is a nonzero operator in $B' \cap pMp$. Take $S \in B' \cap pMp$ such that $TS$ is a nonzero spectral projection $q^{\prime\prime}$ of $T$. Note that the projection $q^{\prime\prime}$ is orthogonal to $q$ and that the formula
$$(B)_1 q^{\prime\prime} = (p-q) \; (B)_1 q' \; (p-q)S \; ,$$
together with Lemma \ref{lem.almost} implies that $(B q^{\prime\prime})_1 \almost N \rtimes \cS$. This contradicts the maximality of the sequence $(q_k)$. So, $\cP$ attains its maximum in $q \in \cP$.

If $u \in \cN_{pMp}(B)$, the fact that $(B)_1 q \almost N \rtimes \cS$ and Lemma \ref{lem.almost} imply that $u (B)_1 q u^* \almost N \rtimes \cS$. But $u (B)_1 q u^* = (B)_1 u q u^*$ and it follows that $uqu^* \leq q$. Hence, $q$ commutes with $Q$. In particular, $q$ commutes with $B$ so that $q \in Q$. Hence, $q \in \cZ(Q)$.

By Lemma \ref{lem.intertwining} it remains to prove that for all $\Sigma \in \cS$ we have $B(p-q) \not\prec N \rtimes \Sigma$. If the contrary would be true, the implication \ref{stat.1}~$\Rightarrow$~\ref{stat.2} in Lemma \ref{lem.approx-vs-prec} provides a nonzero projection $q' \in B' \cap pMp$ such that $q' \leq p-q$ and such that $(Bq')_1 \almost N \rtimes \cS$. This contradicts the maximality of $q$.
\end{proof}

\begin{lemma}\label{lem.intersection}
Assume that $\cS_1$ and $\cS_2$ are two families of subgroups of $\Gamma$. Define the family $\cS$ consisting of the subgroups $\Sigma_1 \cap g \Sigma_2 g^{-1}$, $\Sigma_i \in \cS_i$, $g \in \Gamma$. Let $B \subset pMp$ be a von Neumann subalgebra. If $(B)_1 \almost N \rtimes \cS_i$ for $i=1,2$, then $(B)_1 \almost N \rtimes \cS$.
\end{lemma}
\begin{proof}
It suffices to make the following observation: if $\cF_i \subset \Gamma$ are small relative to $\cS_i$ for $i=1,2$, then $\cF_1 \cap \cF_2$ is small relative to $\cS$ and $P_{\cF_1 \cap \cF_2} = P_{\cF_1} \circ P_{\cF_2}$.
\end{proof}

\section{\boldmath Malleable deformations coming from group $1$-cocycles}\label{sec.dilation}

\subsection{Malleable deformation by a one-parameter group of automorphisms}\label{subsec.deform}

Throughout this section $\pi : \Gamma \recht \cO(H_\R)$ is an orthogonal representation of a countable group $\Gamma$ and $b : \Gamma \recht H_\R$ is a $1$-cocycle, i.e.\ a map satisfying
$$b(gh) = b(g) + \pi(g) b(h) \quad\text{for all}\;\; g,h \in \Gamma \; .$$
The function $\Gamma \recht \C : g \mapsto \|b(g)\|^2$ is conditionally of negative type. Whenever $\Gamma \actson (N,\tau)$ is a trace preserving action and $M = N \rtimes \Gamma$, we get a semigroup $(\vphi_t)_{t > 0}$ of unital tracial completely positive maps
$$\vphi_t : M \recht M : \vphi_t(au_g) = \exp(-t \|b(g)\|^2) \, a u_g \quad\text{for all}\;\; a \in N, g \in \Gamma \; .$$
Following \cite[beginning of Section 3]{Si10} we construct a malleable deformation of $M$ in the sense of Popa (see \cite[Section 6]{Po06a}), i.e.\
a canonical larger finite von Neumann algebra $M \subset \Mtil$ together with a $1$-parameter group of automorphisms $(\al_t)_{t \in \R}$ of $\Mtil$ such that $\vphi_{t^2}(x) = E_M(\al_t(x))$ for all $x \in M$ and $t \in \R$.

Denote by $\Gamma \actson (Z,\eta)$ the Gaussian action associated with $\pi$. Put $D = \rL^\infty(Z)$ and denote by $\tau$ the trace on $D$ given by integration with respect to $\eta$. Denote by $(\si_g)_{g \in \Gamma}$ the Gaussian action viewed as an action by trace preserving automorphisms of $D$. For our purposes it is most convenient to consider $(D,\tau)$ as the unique pair of a von Neumann algebra $D$ with a faithful tracial state $\tau$ such that $D$ is generated by unitary elements $\om(\xi), \xi \in H_\R$, satisfying
\begin{itemize}
\item $\om(0) = 1$, $\om(\xi_1 + \xi_2) = \om(\xi_1) \, \om(\xi_2)$ for all $\xi_1,\xi_2 \in H_\R$ and $\om(\xi)^* = \om(-\xi)$ for all $\xi \in H_\R$,
\item $\tau(\om(\xi)) = \exp(-\|\xi\|^2)$ for all $\xi \in H_\R$.
\end{itemize}
Moreover, $\sigma_g(\om(\xi)) = \om(\pi(g) \xi)$. Note that the linear span of all $\om(\xi)$, $\xi \in H_\R$, is a dense $*$-subalgebra of $D$.

Define $\Mtil = (D \ovt N) \rtimes \Gamma$ where $\Gamma \actson D \ovt N$ is the diagonal action. The formula
$$\al_t(x) = x \;\;\text{for all}\;\; x \in D \ovt N \quad\text{and}\quad \al_t(u_g) = (\om(t b(g)) \ot 1) u_g \;\;\text{for all}\;\; g \in \Gamma$$
provides the required $1$-parameter group of automorphisms.

We finally prove that the deformation $(\al_t)_{t \in \R}$ is $s$-malleable in the sense of \cite[Section 6]{Po06a}. The formula $\beta(\om(\xi)) = \om(-\xi) = \om(\xi)^*$ for all $\xi \in H_\R$, defines a trace preserving automorphism of $D$. This automorphism extends to an automorphism $\beta$ of $\Mtil$ satisfying $\beta(x) = x$ for all $x \in M$. One easily checks that $\beta$ satisfies the $s$-malleability conditions~: $\beta^2 = \id$ and $\beta \circ \al_t = \al_{-t} \circ \beta$ for all $t \in \R$.

\subsection{Two easy inequalities}

We start by proving a version of the \lq transversality inequality\rq\ from \cite[Lemma 2.1]{Po06c}.

\begin{lemma}\label{lem.equivalent}
Denote for $x \in M$, $\delta_t(x) = \al_t(x) - E_M(\al_t(x))$. Then,
$$\|\delta_t(x)\|_2 \leq \|x-\al_t(x)\|_2 \leq \sqrt{2} \|\delta_t(x)\|_2 \quad\text{for all}\;\; x \in M \; , \; t \in \R \; .$$
So, studying convergence of $\al_t \recht \id$ is equivalent to studying convergence of $\delta_t \recht 0$.

Also, if $x \in M$ and $|t|$ decreases, the expressions $\|\delta_t(x)\|_2$ and $\|x - \al_t(x)\|_2$ decrease.
\end{lemma}
\begin{proof}
Define $f_t(g) = \exp(-t^2 \|b(g)\|^2)$ and note that $0 < f_t(g) \leq 1$ for all $g \in \Gamma$ and $t \in \R$. Take $x \in M$ and write $x = \sum_{g \in \Gamma} x_g u_g$ with $x_g \in N$.
It is easy to check that
$$\|\delta_t(x)\|_2^2 = \sum_{g \in \Gamma} (1 - f_t(g)^2) \|x_g\|_2^2 \quad\text{and}\quad \|x-\al_t(x)\|_2^2 = 2 \sum_{g \in \Gamma} (1-f_t(g)) \|x_g\|_2^2 \; .$$
Using the inequalities $1 - f_t(g)^2 \leq 2(1-f_t(g)) \leq 2(1-f_t(g)^2)$, the lemma follows.
\end{proof}

\begin{lemma}\label{noglemma}
For all $x,y \in (\Mtil)_1$ and all $t,\eps > 0$, there exists $s > 0$ such that
$$\|E_M(x \al_t(a) y)\|_2 \leq \|E_M(\al_s(a))\|_2 + \eps \quad\text{for all}\;\; a \in (M)_1 \; .$$
\end{lemma}
\begin{proof}
Fix $x,y \in (\Mtil)_1$ and fix $t,\eps > 0$. By the Kaplansky density theorem choose $x_1,\ldots,x_n$, $y_1,\ldots,y_m \in M$ and $\xi_1,\ldots,\xi_n,\eta_1,\ldots,\eta_m \in H_\R$ such that the elements
$$x_0 := \sum_{i=1}^n x_i \om(\xi_i) \quad\text{and}\quad y_0 := \sum_{j=1}^m \om(\eta_j) y_j$$
satisfy $\|x_0\|, \|y_0\| \leq 1$ and $\|x-x_0\|_2 \leq \eps/4$, $\|y-y_0\|_2 \leq \eps / 4$.

Define $\kappa_1 = \max\{\|x_i\| \mid i = 1,\ldots,n\}$ and $\kappa_2 = \max \{\|y_j\| \mid j=1,\ldots,m\}$. Also put $\gamma_1 = \max\{\|\xi_i\| \mid i=1,\ldots,n\}$ and $\gamma_2 = \max \{\|\eta_j\| \mid j=1,\ldots,m\}$. Choose $\rho > 0$ large enough such that $\exp(-\rho^2) \leq \eps / (4 n \kappa_1 m \kappa_2)$. Then put $\kappa = (\rho + \gamma_1 + \gamma_2)/t$. Finally take $s > 0$ small enough such that $\exp(- 2 s^2 \kappa^2) \geq 1 - (\eps/4)^2$. We claim that this $s$ satisfies the conclusions of the lemma.

To prove the claim define $\cF := \{g \in \Gamma \mid \|b(g)\| \leq \kappa\}$. Whenever $g \in \Gamma - \cF$, we have $\| t b(g)\| \geq \rho + \gamma_1 + \gamma_2$ and hence $\|\xi_i + t b(g) + \pi(g)\eta_j\| \geq \rho$ so that
$$\exp(-\|\xi_i + t b(g) + \pi(g)\eta_j\|^2) \leq \exp(-\rho^2) \leq  \frac{\eps}{4 n \kappa_1 m \kappa_2} \quad\text{for all}\;\; i = 1,\ldots,n \; , \; j=1,\ldots,m \;\; .$$
For any subset $\cG \subset \Gamma$ we denote by $P_\cG$ the orthogonal projection of $\rL^2(M)$ onto the closed linear span of $\{a u_g \mid a \in N, g \in \cG\}$.
We decompose every $a \in M$ in its Fourier expansion $a = \sum_{g \in \Gamma} a_g u_g$ with $a_g \in N$.
Noticing that for all $a \in M$ we have
$$E_M(\om(\xi_i) \al_t(P_{\Gamma-\cF}(a)) \om(\eta_j)) = \sum_{g \in \Gamma - \cF} \exp (-\|\xi_i + t b(g) + \pi(g)\eta_j\|^2) \; a_g u_g \; ,$$
it follows that for all $a \in (M)_1$ we have
$$\|E_M(\om(\xi_i) \al_t(P_{\Gamma-\cF}(a)) \om(\eta_j))\|_2 \leq \frac{\eps}{4 n \kappa_1 m \kappa_2} \; .$$
Multiplying with $x_i$ on the left and with $y_j$ on the right and summing over $i$ and $j$, we conclude that
$$\|E_M(x_0 \al_t(P_{\Gamma-\cF}(a)) y_0)\|_2 \leq \frac{\eps}{4}$$
for all $a \in (M)_1$. Since $\|x_0\|,\|y_0\| \leq 1$, it follows that for all $a \in (M)_1$ we have
$$\|E_M(x_0 \al_t(a) y_0)\|_2 \leq \|P_\cF(a)\|_2 + \frac{\eps}{4} \; .$$
Whenever $g \in \cF$ we have $\|b(g)\| \leq \kappa$ and therefore $\exp(- 2 \|s b(g)\|^2) \geq 1 - (\eps/4)^2$. So, for all $a \in (M)_1$ we get
\begin{align*}
\|P_\cF(a)\|_2^2 &= \sum_{g \in \cF} \|a_g\|_2^2 \leq \sum_{g \in \Gamma} \Bigl(  \exp(- 2 \|s b(g)\|^2) + \frac{\eps^2}{16}\Bigr) \; \|a_g\|_2^2
\\ &= \|E_M(\al_s(a))\|_2^2 + \Bigl(\frac{\eps}{4} \|a\|_2 \Bigr)^2 \leq \Bigl( \|E_M(\al_s(a))\|_2 + \frac{\eps}{4}\Bigr)^2 \; .
\end{align*}
Altogether it follows that
$$\|E_M(x_0 \al_t(a) y_0)\|_2 \leq \|E_M(\al_s(a))\|_2 + \frac{\eps}{2}$$
for all $a \in (M)_1$. Our choice of $x_0$ and $y_0$ then imply the claim.
\end{proof}

\subsection{The maximal projection under which the malleable deformation is uniform}

\begin{lemma}\label{lem.maxuniform-general}
Let $p \in M$ be a projection and $B \subset pMp$ a von Neumann subalgebra. Define $Q \subset pMp$ as the normalizer of $B$ inside $pMp$. The set of projections
$$\cP := \{q_0 \mid q_0 \;\;\text{is a projection in}\;\; B' \cap pMp \;\;\text{and}\;\; \al_t \recht \id \;\;\text{in \two uniformly on}\; (B q_0)_1 \}$$
attains its maximum in a unique projection $q \in \cP$. This projection $q$ belongs to $\cZ(Q)$.
\end{lemma}
\begin{proof}
Let $q_k \in B' \cap pMp$ be a maximal orthogonal sequence of nonzero elements of $\cP$. Put $q = \sum_k q_k$. It is easy to check that $q \in \cP$. We claim that $q$ is the maximum of $\cP$. So assume that $e \in \cP$. We have to prove that $e \leq q$. If this is not the case, $T:= (p-q) e (p-q)$ is nonzero. Since $\al_t \recht \id$ uniformly on $(B)_1 e$, the same is true on
$$(B)_1 T = (p-q) \; (B)_1 e \; (p-q) \; .$$
Since $T$ is nonzero we can find a bounded operator $S \in B' \cap pMp$ such that $T S$ is a nonzero spectral projection $f$ of $T$. It follows that $\al_t \recht \id$ uniformly on $(B)_1 f = (B)_1 T \; S$. This contradicts the maximality of the sequence $(q_k)$.

We finally prove that $q \in \cZ(Q)$. Let $u \in \cN_{pMp}(B)$. Since $\al_t \recht \id$ uniformly on $(B)_1 q$, the same is true on $u \; (B)_1 q \; u^*$, which equals $(B)_1 \; uqu^*$. Hence $uqu^* \leq q$ which means that $uqu^* = q$ for all $u \in \cN_{pMp}(B)$. So, $q \in Q' \cap pMp = \cZ(Q)$.
\end{proof}

Assume that $B \subset pMp$ is a von Neumann subalgebra and let $q \in B' \cap pMp$ be as in the previous lemma, the largest projection such that $\al_t \recht \id$ uniformly on $(B)_1 q$. In good cases, it follows that on $(B)_1 (p-q)$ the convergence is the furthest possible from being uniform, in the sense that we can find a sequence of unitaries $(w_n)$ in $\cU(B)$ such that for every $t > 0$ we have that $\|E_M(\al_t(w_n (p-q)))\|_2 \recht 0$ as $n \recht \infty$. These good cases correspond to $\pi$ being mixing relative to a family $\cS$ of subgroups of $\Gamma$ such that the $1$-cocycle $b$ is bounded on every $\Sigma \in \cS$, see Proposition \ref{prop.maxuniform} below.

To establish uniform convergence of $\al_t$ on $Bq$, it often suffices to prove uniform convergence on $r \cG r$ where $r \leq q$ is a smaller projection and $\cG$ is a group of unitaries generating $B$. The precise result goes as follows.

\begin{proposition}\label{prop.extenduniform}
Let $p \in M$ be a projection and $B \subset pMp$ a von Neumann subalgebra generated by a group of unitaries $\cG \subset \cU(B)$. Let $r \in pMp$ be any projection and assume that $\al_t \recht \id$ in \two uniformly on the set $r\cG r$.

Define $Q$ as the normalizer of $B$ inside $pMp$ and denote by $q$ the smallest projection in $\cZ(Q)$ satisfying $r \leq q$. Then, $\al_t \recht \id$ in \two uniformly on the unit ball of $Bq$.
\end{proposition}
\begin{proof}
For $x \in M$ we write $\delta_t(x) = \al_t(x) - E_M(\al_t(x))$. By Lemma \ref{lem.equivalent} we have that $\delta_t \recht 0$ in \two uniformly on $r \cG r$. Define $T := E_{B' \cap pMp}(r)$. We claim that $\delta_t \recht 0$ in \two uniformly on $T \cG T$. To prove this claim, choose $\eps > 0$. Note that $T$ coincides with the unique element of minimal \two in the $\|\,\cdot\,\|_2$-closed convex hull of $\{uru^* \mid u \in \cG\}$. So we can take $u_1,\ldots,u_n \in \cG$ and $\lambda_1,\ldots,\lambda_n \in [0,1]$ such that $\sum_i \lambda_i = 1$ and such that
$$T_0:= \sum_{i=1}^n \lambda_i u_i r u_i^*$$
satisfies $\|T-T_0\|_2 \leq \eps$. Note that $\|T_0\| \leq 1$.

Choose $t > 0$ small enough such that $\|\delta_t(r u r)\|_2 \leq \eps$ for all $u \in \cG$ and such that $\|\al_t(u_i) - u_i\|_2 \leq \eps$ for all $i=1,\ldots,n$. The latter implies that $\|\delta_t(u_i x u_j^*) - u_i \delta_t(x) u_j^*\|_2 \leq 4\eps$ for all $x \in (M)_1$. Whenever $u \in \cG$, we have $u_i^* u u_j \in \cG$ and hence
$$\|\delta_t(u_i r u_i^* \; u \; u_j r u_j^*)\|_2 \leq \|u_i \; \delta_t(r \; u_i^* u u_j \; r) \; u_j^*\|_2 + 4 \eps \leq 5 \eps \; .$$
Taking convex combinations it follows that $\|\delta_t(T_0 u T_0)\|_2 \leq 5 \eps$ for all $u \in \cG$. Since $\|T - T_0\|_2 \leq \eps$ and both $\|T\|,\|T_0\| \leq 1$, it follows that $\|\delta_t(T u T)\|_2 \leq 7 \eps$ for all $u \in \cG$, hence proving the claim.

Since $T$ commutes with the elements in $\cG$ and since for all $u \in \cG$ we have $\|\delta_t(u T^2) - \delta_t(u)T^2\|_2 \leq 2\|\al_t(T^2) - T^2\|_2$, it follows that $\|\delta_t(u)T^2\|_2 \recht 0$ uniformly in $u \in \cG$. For every $\delta > 0$ define the spectral projection $q_\delta = \chi_{(\delta,+\infty)}(T^2)$ and take bounded operators $T_\delta \in B' \cap pMp$ such that $T^2 T_\delta = q_\delta$. It follows that for every $\delta > 0$, we have that $\|\delta_t(u)q_\delta\|_2 \recht 0$ uniformly in $u \in \cG$. Denote by $q_0$ the support projection of $T$. When $\delta \recht 0$ we have $\|q_0-q_\delta\|_2 \recht 0$. Hence, $\|\delta_t(u) q_0\|_2 \recht 0$ uniformly in $u \in \cG$. So, $\al_t \recht \id$ in \two uniformly on $\cG q_0$.

Note that $r \leq q_0$ and that actually $q_0$ precisely is the smallest projection in $B' \cap pMp$ satisfying $r \leq q_0$. By Lemma \ref{lem.maxuniform-general} it remains to prove that $\al_t \recht \id$ in \two uniformly on the unit ball of $B q_0$. Choose $\eps > 0$. Take $t > 0$ such that $\|uq_0 - \al_t(uq_0)\|_2 \leq \eps$ for all $u \in \cG$. In particular $\|q_0 - \al_t(q_0)\|_2 \leq \eps$. Define $v \in \Mtil$ as the element of smallest \two in the $\|\,\cdot\,\|_2$-closed convex hull of $\{u q_0 \al_t(q_0 u^*) \mid u \in \cG\}$. Note that $b v = v \al_t(b)$ for all $b \in Bq_0$ and that $\|v-q_0\|_2 \leq \eps$. Whenever $b \in (B)_1 q_0$ we have
$$\|\al_t(b) - v \al_t(b)\|_2 = \|(\al_t(q_0) - v) \al_t(b)\|_2 \leq \|\al_t(q_0) - v\|_2 \leq \|\al_t(q_0) - q_0\|_2 + \|q_0 - v\|_2 \leq 2 \eps \; .$$
We also have for all $b \in (B)_1 q_0$ that $\|b v - b\|_2 \leq \|b (v-q_0)\|_2 \leq \eps$. Since $bv = v \al_t(b)$ we conclude that $\|\al_t(b) - b\|_2 \leq 3 \eps$ for all $b \in (B)_1 q_0$. This ends the proof of the proposition.
\end{proof}

\subsection{Cocycles into representations that are weakly contained in the regular representation}

\begin{lemma}\label{lem.regular}
Assume that $\pi : \Gamma \recht \cO(H_\R)$ is weakly contained in the regular representation and that $N$ is amenable. Then the bimodule $\bim{M}{\rL^2(\Mtil \ominus M)}{M}$ is weakly contained in the coarse $M$-$M$-bimodule.
\end{lemma}
\begin{proof}
Whenever $\eta : \Gamma \recht \cU(\cK)$ is a unitary representation, we define the $M$-$M$-bimodule $\cH^\eta$ on the Hilbert space $\rL^2(M) \ot \cK$ with bimodule structure
$$(a u_g) \cdot (x \ot \xi) \cdot (b u_h) = a u_g x b u_h \ot \eta(g) \xi \quad\text{for all}\quad a,b \in N, g,h \in \Gamma, x \in M, \xi \in \cK \; .$$
If $\eta$ is the regular representation, one checks that $\cH^\eta$ is unitarily equivalent with $\rL^2(M) \ot_N \rL^2(M)$. Because $N$ is amenable we conclude that $\cH^\eta$ is weakly contained in the coarse bimodule whenever $\eta$ is weakly contained in the regular representation.

One checks that $\bim{M}{\rL^2(\Mtil \ominus M)}{M}$ is unitarily equivalent to the direct sum of $\cH^\eta$ where $\eta$ runs through all the symmetric tensor powers of the complexified $\pi$. Since all these symmetric tensor powers are weakly contained in the regular representation, the lemma follows.
\end{proof}

Following Ozawa \cite{Oz03}, a II$_1$ factor $M$ is called solid if $A' \cap M$ is amenable for every diffuse von Neumann subalgebra $A \subset M$. Ozawa proved in \cite{Oz03} that all group von Neumann algebras of icc hyperbolic groups are solid. In particular the free group factors $\rL \F_n$, $n \geq 2$, are solid.
Although this is not needed in the rest of this paper, we reprove Peterson's \cite[Theorem 1.3]{Pe06} on the solidity of the group factors $\rL \Gamma$ when $\Gamma$ is an icc group that admits a proper\footnote{Recall that a $1$-cocycle $b : \Gamma \recht H_\R$ is called proper if $\|b(g)\| \recht \infty$ whenever $g \recht \infty$.} $1$-cocycle into a multiple of the regular representation. Prior to this Popa had given in \cite{Po06b} a new proof for the solidity of the free group factors $\rL \F_n$ using the malleable deformation coming from the word length. Our proof is very close to Popa's, replacing the word length by a proper $1$-cocycle.

It is shown in \cite{CSV07} that for every finite group $H$ and $n \geq 2$, the wreath product $\Gamma = H \wr \F_n := H^{(\F_n)} \rtimes \F_n$ admits a so-called proper wall structure with the stabilizer of every wall being amenable. Hence $\Gamma$ admits a proper $1$-cocycle into an orthogonal representation that is weakly contained in the regular representation. In particular $\rL \Gamma$ is solid by Theorem \ref{thm.solid} below. A more general result is proven by Ozawa in \cite[Corollary 4.5]{Oz04} implying that all wreath products $H \wr \Gamma$ with $H$ amenable and $\Gamma$ hyperbolic, have a solid group von Neumann algebra.

\begin{theorem}[{\cite[Theorem 1.3]{Pe06}}]\label{thm.solid}
Let $\Gamma$ be an icc group that admits a proper $1$-cocycle into an orthogonal representation $\pi$ that is weakly contained in the regular representation. Then $\rL \Gamma$ is solid.
\end{theorem}
\begin{proof}
Let $A \subset M$ be a diffuse von Neumann subalgebra and assume by contradiction that $A' \cap M$ is nonamenable. So we can take a nonzero projection $p \in \cZ(A' \cap M)$ such that $P := (A' \cap M)p$ has no amenable direct summand.

Let $b : \Gamma \recht H_\R$ be a proper $1$-cocycle into the orthogonal representation $\pi : \Gamma \recht \cO(H_\R)$ that is weakly contained in the regular representation. Consider the one-parameter group of automorphisms $(\al_t)$ of $\Mtil = D \rtimes \Gamma$ as above. By Lemma \ref{lem.regular} the bimodule $\bim{M}{\rL^2(\Mtil \ominus M)}{M}$ is weakly contained in the coarse $M$-$M$-bimodule. By assumption $P$ has no amenable direct summand and we just observed that $\bim{P}{\rL^2(p\Mtil p \ominus pMp)}{P}$ is weakly contained in the coarse $P$-$P$-bimodule. As we explain for the sake of completeness in Remark \ref{rem.amenable} below, it follows that for every $\eps > 0$, there exists a finite subset $\cF \subset (P)_1$ and a $\delta > 0$ such that
\begin{equation}\label{eq.nonamenablecond}
\text{if}\;\; x \in p(\Mtil \ominus M)p \;\;\text{satisfies}\;\; \|x\| \leq 2 \;\;\text{and}\;\; \|[x,y]\|_2 \leq \delta \;\;\text{for all}\;\; y \in \cF \;\;\text{then}\;\; \|x\|_2 \leq \eps \; .
\end{equation}
We prove that $\al_t \recht \id$ uniformly on the unit ball of $Ap$.
Choose $\eps > 0$. Take a finite subset $\cF \subset (P)_1$ and a $\delta > 0$ such that \eqref{eq.nonamenablecond} holds. Put $\delta_0 = \min \{\delta/12,\eps\}$. Take $t > 0$ small enough such that $\|\al_t(p) - p\|_2 \leq \delta_0$ and $\|\al_t(y) - y\|_2 \leq \delta_0$ for all $y \in \cF$. We claim that $\|\al_t(ap) - ap\|_2 \leq 5\sqrt{2}\eps$ for all $a \in (A)_1$. Once this claim is proven, we have shown that $\al_t \recht \id$ uniformly on the unit ball of $Ap$.

To prove the claim, fix $a \in (A)_1$. As before write $\delta_t(d) = \al_t(d) - E_M(\al_t(d))$ for all $d \in M$. Define $x = p \delta_t(a) p$. Since $\|\al_t(p) - p\|_2 \leq \delta_0$, it follows that $\|p \al_t(a) p - \al_t(ap) \|_2 \leq 2\delta_0$ and hence $\|x-\delta_t(ap)\|_2 \leq 4\delta_0$.
Since $\al_t(ap)$ and $\al_t(y)$ commute for all $y \in \cF$, it follows that $\|[\al_t(ap),y]\|_2 \leq 2\delta_0$ and hence $\|[\delta_t(ap),y]\|_2 \leq 4\delta_0$ for all $y \in \cF$. So, $\|[x,y]\|_2 \leq 12\delta_0 \leq \delta$ for all $y \in \cF$. Since $\|x\| \leq 2$ we conclude from \eqref{eq.nonamenablecond} that $\|x\|_2 \leq \eps$. Hence, $\|\delta_t(ap)\|_2 \leq \eps + 4\delta_0 \leq 5\eps$. By Lemma \ref{lem.equivalent} we get that $\|\al_t(ap) - ap\|_2 \leq 5 \sqrt{2}\eps$.

So $\al_t \recht \id$ uniformly on the unit ball of $Ap$. Using the facts that $Ap$ is diffuse and that the $1$-cocycle $b$ is proper, we finally derive a contradiction. Since $\al_t \recht \id$ uniformly on $\cU(Ap)$, we can fix $t > 0$ such that $\|E_M(\al_t(wp))\|_2 \geq \|p\|_2 / 2$ for all $w \in \cU(A)$. Since $A$ is diffuse, we can take a sequence of unitaries $w_n \in \cU(A)$ tending to zero weakly. We prove that $\|E_M(\al_t(w_n p))\|_2 \recht 0$ as $n \recht \infty$. Let $w_n p = \sum_{g \in \Gamma} \lambda^n_g u_g$ be the Fourier expansion of $w_n p$, with $\lambda^n_g \in \C$. Since $w_n p \recht 0$ weakly, for every fixed $g \in \Gamma$ we have $\lambda^n_g \recht 0$ as $n \recht \infty$. Fix $\eps > 0$. Take a finite subset $\cF \subset \Gamma$ such that $\exp(- 2 t^2 \|b(g)\|^2) < \eps$ for all $g \in \Gamma-\cF$. Take $n_0 \in \N$ such that $\sum_{g \in \cF} |\lambda^n_g|^2 < \eps$ for all $n \geq n_0$. So for all $n \geq n_0$ we have
$$\|E_M(\al_t(w_n p))\|_2^2 = \sum_{g \in \Gamma} \exp(-2 t^2 \|b(g)\|^2) |\lambda^n_g|^2 \leq \sum_{g \in \cF} |\lambda^n_g|^2 + \sum_{g \in \Gamma - \cF} \eps |\lambda^n_g|^2 < \eps + \eps \|w_n p\|_2^2 \leq 2 \eps \; .$$
Hence we have shown that $\|E_M(\al_t(w_n p))\|_2 \recht 0$ as $n \recht \infty$ which is a contradiction with the fact that $\|E_M(\al_t(wp))\|_2 \geq \|p\|_2 / 2$ for all $w \in \cU(A)$.
\end{proof}

\begin{remark}\label{rem.amenable}
The following is a detailed explanation for \eqref{eq.nonamenablecond}.
Let $(P,\tau)$ be a tracial von Neumann algebra and $\bim{P}{\cH}{P}$ a $P$-$P$-bimodule. Assume that $\xi_n \in \cH$ is a sequence of vectors and $\kappa,\eps > 0$ are positive numbers such that
\begin{align*}
& \|\xi_n\| \geq \eps \;\;\text{for all}\;\; n \in \N \;\; , \quad \langle x \xi_n,\xi_n \rangle \leq \kappa \tau(x) \;\;\text{for all}\;\; x \in P^+ \quad\text{and}\\
& \lim_n \|x \xi_n - \xi_n x \| = 0 \;\;\text{for all}\;\; x \in P \; .
\end{align*}
We claim that there is a nonzero central projection $z \in \cZ(P)$ such that $z \cH z$ weakly contains the trivial bimodule $\bim{Pz}{\rL^2(Pz)}{Pz}$. By our assumptions we find $T_n \in P^+$ such that $\|T_n\| \leq \kappa$ and $\langle x \xi_n ,\xi_n \rangle = \tau(x T_n)$ for all $x \in P$. After a passage to a subsequence we may assume that $T_n \recht T$ weakly. Since $\|\xi_n\| \geq \eps$ for all $n \in \N$ we have $\tau(T) \geq \eps^2$. So $T \neq 0$. Because $\lim_n \|x \xi_n - \xi_n x \| =0$ for all $x \in P$, it follows that $T \in \cZ(P)$. Since $T$ is positive, take a nonzero central projection $z \in \cZ(P)$ and an element $S \in \cZ(P)^+$ such that $S^2 T = z$. Write $\eta_n = S \xi_n$. It follows that
$$\langle x \eta_n y , a \eta_n b \rangle \recht \tau(x y b^* a^*) = \langle xy,ab\rangle \quad\text{for all}\;\; x,y,a,b \in Pz \; .$$
This proves our claim.
\end{remark}

\subsection{Cocycles into representations that are mixing relative to a family of subgroups}

Assume that $\pi$ is mixing relative to a family $\cS$ of subgroups of $\Gamma$. As before we say that a subset $\cF \subset \Gamma$ is small relative to $\cS$ if $\cF$ can be written as a finite union of $g\Sigma h$, $g,h \in \Gamma$, $\Sigma \in \cS$. Denote by $P_\cF$ the orthogonal projection of $\rL^2(M)$ onto the closed linear span of $\{a u_g \mid a \in N, g \in \cF\}$.

We prove the following technical lemma. It will be useful in Section \ref{sec.normalizing}, but also when proving Proposition \ref{prop.maxuniform} and Theorem \ref{thm.uniform-normalizer} below. The final statement of Lemma \ref{lemma1} \lq controls\rq\ the normalizer of certain von Neumann subalgebras $B \subset p\Mtil p$. The usage of mixing techniques to control normalizers goes back to Popa's \cite[Section 3]{Po03}.

\begin{lemma}\label{lemma1}
Assume that $\pi$ is mixing relative to a family $\cS$ of subgroups of $\Gamma$. Use the notation $P_\cF$ as explained before the lemma.

Let $K \subset D \ominus \C 1$ be a finite dimensional vector subspace. Denote by $Q_K$ the orthogonal projection of $\rL^2(\Mtil)$ onto the closed linear span of $(d \ot a)u_g$, $d \in K$, $a \in N$, $g \in \Gamma$. Note that $Q_K$ is right $M$-modular.

For every finite dimensional subspace $K \subset D \ominus \C 1$, every $x \in (\Mtil)_1$ and every $\eps > 0$, there exists a subset $\cF \subset \Gamma$ that is small relative to $\cS$ such that
$$\|Q_K(v x)\|_2 \leq \|P_\cF(v)\|_2 + \eps \quad\text{for all}\quad v \in (M)_1 \; .$$
In particular, if $p \in M$ is a projection and $B \subset pMp$ is a von Neumann subalgebra satisfying $B \not\prec N \rtimes \Sigma$ for all $\Sigma \in \cS$, then the normalizer of $B$ inside $p \Mtil p$ is contained in $pMp$.
\end{lemma}
\begin{proof}
Fix the finite dimensional subspace $K \subset D \ominus \C 1$, the element $x \in (\Mtil)_1$ and $\eps > 0$. By the Kaplansky density theorem we can take $d_1,\ldots,d_n \in D$ and $z_1,\ldots,z_n \in M$ such that the element
$$x_0 := \sum_{i =1}^n (d_i \ot 1) z_i$$
satisfies $\|x_0\| \leq 1$ and $\|x-x_0\|_2 \leq \eps/2$. Put $\kappa = \max\{\|z_1\|,\ldots,\|z_n\|\}$. Note that $(\si_g)_{g \in \Gamma}$ viewed as a unitary representation of $\Gamma$ on $\rL^2(D \ominus \C 1)$, is mixing relative to $\cS$. So we can take a subset $\cF \subset \Gamma$ that is small relative to $\cS$ such that
$$\|Q_K(\si_g(d_i))\|_2 \leq \frac{\eps}{2n\kappa} \quad\text{for all}\quad g \in \Gamma - \cF \; .$$
We claim that $\cF$ satisfies the conclusion of the lemma. So, take $v \in (M)_1$. We have to prove that $\|Q_K(v x)\|_2 \leq \|P_\cF(v)\|_2 + \eps$.

Since $\|vx - vx_0\|_2 \leq \eps/2$, also $\|Q_K(vx) - Q_K(v x_0)\|_2 \leq \eps/2$ and it suffices to prove that $\|Q_K(v x_0)\|_2 \leq \|P_\cF(v)\|_2 + \eps/2$. Let $v = \sum_{g \in \Gamma} v_g u_g$, with $v_g \in N$, be the Fourier expansion of $v$. A direct computation yields that
$$Q_K(v x_0) = \sum_{i=1}^n \sum_{g \in \Gamma} (Q_K(\si_g(d_i)) \ot v_g) u_g z_i \; .$$
For every $i = 1,\ldots,n$, we have
\begin{align*}
\Bigl\| \sum_{g \in \Gamma - \cF} (Q_K(\si_g(d_i)) \ot v_g) u_g z_i \Bigr\|_2^2 & \leq \Bigl\| \sum_{g \in \Gamma - \cF} (Q_K(\si_g(d_i)) \ot v_g) u_g \Bigr\|_2^2 \; \|z_i\|^2 \\
& = \|z_i\|^2 \; \sum_{g \in \Gamma - \cF} \|Q_K(\si_g(d_i))\|_2^2 \; \|v_g\|_2^2 \\
& \leq \Bigl(\frac{\eps \|z_i\|}{2n\kappa}\Bigr)^2 \; \sum_{g \in \Gamma-\cF} \|v_g\|_2^2 \leq \Bigl(\frac{\eps}{2n}\Bigr)^2 \; .
\end{align*}
Hence,
$$\Bigl\| \sum_{i=1}^n \sum_{g \in \Gamma - \cF} (Q_K(\si_g(d_i)) \ot v_g) u_g z_i \Bigr\|_2 \leq \frac{\eps}{2} \;.$$
On the other hand,
$$\sum_{i =1}^n \sum_{g \in \cF} (Q_K(\si_g(d_i)) \ot v_g) u_g z_i  = Q_K(P_\cF(v) x_0)$$
and $\|Q_K(P_\cF(v) x_0)\|_2 \leq \|P_\cF(v) x_0\|_2 \leq \|P_\cF(v)\|_2$.
Altogether we have shown that $\|Q_K(v x_0)\|_2 \leq  \|P_\cF(v)\|_2 + \eps/2$.

Finally assume that $B \subset pMp$ is a von Neumann subalgebra satisfying $B \not\prec N \rtimes \Sigma$ for all $\Sigma \in \cS$. Lemma \ref{lem.intertwining} provides a net $(v_i)$ of unitaries in $B$ such that $\|P_\cF(v_i)\|_2 \recht 0$ for every subset $\cF \subset \Gamma$ that is small relative to $\cS$. Let $x \in \cN_{p\Mtil p}(B)$. We have to prove that $x \in M$. Let $K \subset D \ominus \C 1$ be an arbitrary finite dimensional subspace. It suffices to prove that $Q_K(x) = 0$. Let $\eps > 0$. Take a subset $\cF \subset \Gamma$ that is small relative to $\cS$ and such that
$$\|Q_K(v x)\|_2 \leq \|P_\cF(v)\|_2 + \eps$$
for all $v \in (M)_1$. Taking $v=v_i$ and using that $Q_K$ is right $M$-modular, it follows that
$$\|Q_K(x)\|_2 = \|Q_K(x) \; x^* v_i x\|_2 = \|Q_K(v_i x)\|_2 \leq \|P_\cF(v_i)\|_2 + \eps \recht \eps \; .$$
Hence $\|Q_K(x)\|_2 \leq \eps$ for every $\eps > 0$. So $Q_K(x) = 0$.
\end{proof}

\begin{proposition}\label{prop.maxuniform}
Assume that $\pi$ is mixing relative to a family $\cS$ of subgroups of $\Gamma$ and that $b$ is bounded on every $\Sigma \in \cS$.
Let $p \in M$ be a projection and $B \subset pMp$ a von Neumann subalgebra generated by a group of unitaries $\cG \subset \cU(B)$. Denote by $Q$ the normalizer of $B$ inside $pMp$.

Let $q \in \cZ(Q)$ be the maximal projection given by Lemma \ref{lem.maxuniform-general} such that $\al_t \recht \id$ in \two uniformly on the unit ball of $B q$. There exists a sequence of unitaries $(w_n)$ in $\cG$ such that for every $t > 0$ we have that $\|E_M(\al_t(w_n(p-q)))\|_2 \recht 0$ as $n \recht \infty$.
\end{proposition}
\begin{proof}
By Proposition \ref{prop.maxintertwine} let $q_0 \in \cZ(Q)$ be the maximal projection such that $(B q_0)_1 \almost N \rtimes \cS$. Since $b$ is bounded on every subgroup $\Sigma \in \cS$ and hence on every subset $\cF \subset \Gamma$ that is small relative to $\cS$, one checks easily that $\al_t \recht \id$ in \two uniformly on the unit ball of $B q_0$. So in order to prove the proposition we may replace $p$ by $p-q_0$ and $B$ by $B(p-q_0)$ and assume that $B \not\prec N \rtimes \Sigma$ for all $\Sigma \in \cS$.

Assume that there is no sequence of unitaries $(w_n)$ in $\cG$ such that $\|E_M(\al_t(w_n(p-q)))\|_2 \recht 0$ for every $t > 0$. So, we find $s,\delta > 0$ such that $\|E_M(\al_s(b(p-q)))\|_2 \geq \delta$ for all $b \in \cG$. Put $t = \sqrt{2} s$. This means that
$$\tau(b^*(p-q) \al_{t}((p-q)b)) = \|E_M(\al_s(b(p-q)))\|_2^2 \geq \delta^2 \quad\text{for all}\;\; b \in \cG \; .$$
Hence the element $w \in \Mtil$ of minimal \two in the weakly closed convex hull of $$\{b^*(p-q) \al_{t}((p-q)b) \mid b \in \cG\}$$ is nonzero and satisfies $b w = w \al_t(b)$ for all $b \in B$ and $w^* = \al_t(\be(w))$. Here $\beta$ denotes the automorphism of $\Mtil$ satisfying $\beta(x) = x$ for all $x \in M$ and $\beta(\om(\xi)) = \om(\xi)^*$ for all $\xi \in H_\R$. Recall that $\beta \circ \al_t = \al_{-t} \circ \beta$. Also, by construction $q w = 0$. We shall prove that the support projection $r$ of $w w^*$ satisfies $r \leq q$, hence reaching a contradiction. So we have to prove that $r \in B' \cap pMp$ and that $\al_t \recht \id$ uniformly on $(B)_1 r$.

Let $v$ be the polar part of $w$. It follows that $vv^* = r$ and $b v = v \al_t(b)$ for all $b \in B$. It also follows that $v^* v$ is the support projection of $w^* w$. Since $w^* = \al_t(\beta(w))$, we have $w^* w = \al_t(\beta(ww^*))$ and conclude that $v^* v = \al_t(\beta(r))$.
By Lemma \ref{lemma1} and because $B \not\prec N \rtimes \Sigma$ for all $\Sigma \in \cS$, we get that $B' \cap p\Mtil p = B' \cap pMp$. So, $r \in B' \cap pMp$. Since $\beta(x) = x$ for all $x \in M$, we have in particular that $\beta(r) = r$. We have seen above that $v^* v = \al_t(\beta(r))$ and hence $v^* v = \al_t(r)$.

Choose $1 \geq \eps > 0$. By Lemma \ref{noglemma} take $t' > 0$ such that
$$\|E_M(v \al_t(a) v^*)\|_2 \leq \|E_M(\al_{t'}(a))\|_2 + \eps \quad\text{for all}\;\; a \in (M)_1 \; .$$
Applying this to $a = b r$ for $b \in (B)_1$ it follows that $\|br\|_2 - \|E_M(\al_{t'}(br)\|_2 \leq \eps$ for all $b \in (B)_1$. So we have shown that $\|br\|_2 - \|E_M(\al_{t}(br))\|_2 \recht 0$ uniformly in $b \in (B)_1$. This precisely means that $\al_t \recht \id$ in \two uniformly on $(B)_1 r$.
\end{proof}

\subsection{\boldmath From uniform convergence on $B$ to uniform convergence on the normalizer}

In \cite[Theorem 4.5]{Pe06} and \cite[Theorem 2.5]{CP10} Peterson and Chifan-Peterson prove that if $B \subset pMp$ is a von Neumann subalgebra such that $\al_t \recht \id$ uniformly on the unit ball of $B$ and such that for all $\Sigma \in \cS$ we have $B \not\prec N \rtimes \Sigma$, then $\al_t \recht \id$ uniformly on the unit ball of the normalizer of $B$. They prove this theorem using the technology of unbounded derivations. We repeat their proof, but using the malleable dilation technology from paragraph \ref{subsec.deform}, which makes things slightly easier.

\begin{theorem}[{\cite[Theorem 4.5]{Pe06} and \cite[Theorem 2.5]{CP10}}]\label{thm.uniform-normalizer}
Assume that $\pi$ is mixing relative to a family $\cS$ of subgroups of $\Gamma$.
Let $p \in M$ be a projection and $B \subset pMp$ a von Neumann subalgebra. Let $r \in B' \cap pMp$ be a projection and make the following assumptions.
\begin{itemize}
\item If $t \recht 0$ then $\al_t \recht \id$ in \two uniformly on $(Br)_1$.
\item For every $\Sigma \in \cS$ we have $Br \not\prec N \rtimes \Sigma$.
\end{itemize}
Denote by $Q$ the normalizer of $B$ inside $pMp$.
Define $q$ as the smallest projection in $\cZ(Q)$ satisfying $r \leq q$. Then $\al_t \recht \id$ in \two uniformly on $(Qq)_1$.
\end{theorem}

\begin{proof}
From Lemma \ref{lem.intertwining} and Proposition \ref{prop.maxintertwine} it follows that $Bq \not\prec N \rtimes \Sigma$ for all $\Sigma \in \cS$ and that we can take a net of unitaries $(v_i)_{i \in I}$ in $\cU(B)$ such that $\lim_i \|P_\cF(v_i q)\|_2 = 0$ for every subset $\cF \subset \Gamma$ that is small relative to $\cS$. By Lemma \ref{lem.maxuniform-general} we have $\al_t \recht \id$ in \two uniformly on the unit ball of $Bq$.

For every $x \in M$, put $\delta_t(x) = \al_t(x) - E_M(\al_t(x))$. By Lemma \ref{lem.equivalent} and Proposition \ref{prop.extenduniform}
it suffices to prove that $\delta_t \recht 0$ in \two uniformly on $\cN_Q(B)q$. Choose $\eps > 0$. Put $\delta = \eps^2/(4 \tau(q))$. Take $t > 0$ such that
$\|\al_t(bq) - bq \|_2 \leq \delta \|q\|_2$ for all $b \in \cU(B)$.
We show that $\|\delta_t(uq)\|_2 \leq \eps$ for all $u \in \cN_Q(B)$, hence proving the claim. To prove this statement, fix $u \in \cN_Q(B)$. Using the notation in Lemma \ref{lemma1} take a finite dimensional subspace $K \subset D \ominus \C1$ such that $\|(1-Q_K)\delta_t(uq)\|_2 \leq \delta \|q\|_2$. By Lemma \ref{lemma1} take a subset $\cF \subset \Gamma$ that is small relative to $\cS$ and such that $\|Q_K(v \al_t(u))\|_2 \leq \|P_\cF(v)\|_2 + \delta \|q\|_2$ for all $v$ in the unit ball of $M$.

Write $d_i := u^* v_i^* u$ and note that $d_i \in \cU(B)$. By construction $uq = v_i q \, u \, d_i q$ for all $i \in I$. By our choice of $t$ we have
$\|\al_t(v_i q) - v_iq \|_2 \leq \delta \|q\|_2$ and $\|\al_t(d_i q) - d_i q\|_2 \leq \delta \|q\|_2$ for all $i \in I$. Using the fact that $Q_K$ is right $M$-modular, it
follows that for all $i \in I$,
\begin{align*}
\|\delta_t(uq)\|_2^2 &= \langle \al_t(uq),\delta_t(uq)\rangle = \langle \al_t(v_iq \, u \, d_iq) , \delta_t(uq) \rangle \\
&\leq |\langle v_i q \; \al_t(u) \; d_i q , \delta_t(uq) \rangle| + 2\delta \tau(q) \leq |\langle v_i q \; \al_t(u) \; d_i q, Q_K(\delta_t(uq)) \rangle| +3\delta \tau(q) \\
&= |\langle Q_K(v_i q \; \al_t(u) \; d_i q) , \delta_t(uq) \rangle| +3\delta \tau(q) \leq \|Q_K(v_i q \; \al_t(u)) \; d_i q\|_2 \; \|q\|_2 +3\delta \tau(q) \\
&\leq \|Q_K(v_i q \; \al_t(u))\|_2 \; \|q\|_2 +3\delta \tau(q) \leq \|P_\cF(v_i q)\|_2 \; \|q\|_2 + 4 \delta \tau(q) \; .
\end{align*}
Taking the limit over $i \in I$ it follows that $\|\delta_t(uq)\|_2^2 \leq 4 \delta \tau(q) = \eps^2$. We have shown that $\delta_t \recht 0$ in \two uniformly on $\cN_Q(B)q$.
\end{proof}

\section{Deformations must be uniform when they are uniform on enough normalizing unitaries}\label{sec.normalizing}

Let $\Gamma \actson (N,\tau)$ be a trace preserving action and put $M = N \rtimes \Gamma$. Assume that $\pi : \Gamma \recht \cO(H_\R)$ is an orthogonal representation that is mixing relative to a family $\cS$ of subgroups of $\Gamma$. Let $b : \Gamma \recht H_\R$ be a $1$-cocycle w.r.t.\ $\pi$ and assume that $b$ is bounded on every $\Sigma \in \cS$. As above we call a subset $\cF \subset \Gamma$ small relative to $\cS$ if $\cF$ can be written as a finite union of subsets of the form $g \Sigma h$, $g, h \in \Gamma$, $\Sigma \in \cS$. Whenever $\cF \subset \Gamma$ we denote by $P_\cF$ the orthogonal projection of $\rL^2(M)$ onto the closed linear span of $\{a u_g \mid a \in N, g \in \Gamma\}$.

Consider the algebra $\Mtil \supset M$ together with the $1$-parameter group of automorphisms $(\al_t)$ of $\Mtil$ as in paragraph \ref{subsec.deform}.
We say that $\al_t \recht \id$ in \two uniformly on the tail of a net $(v_i)_{i \in I}$ if for every $\eps > 0$, there exists a $t_0 > 0$ and $i_0 \in I$ such that $\|v_i - \al_t(v_i)\|_2 < \eps$ for all $i \geq i_0$ and all $|t| < t_0$.

In this section we prove the following result: if $B \subset pMp$ is an abelian von Neumann subalgebra that is normalized by \lq enough\rq\ unitaries $(v_i)$ such that $\al_t \recht \id$ uniformly on the tail of $(v_i)$, then $\al_t \recht \id$ uniformly on the unit ball of $B$. In the special where $\cS$ consists of normal subgroups of $\Gamma$ and using the technology of unbounded derivations and ultrapowers, this theorem was first proven by Chifan-Peterson in \cite[Theorem 3.2]{CP10} and several of the ideas go back to Peterson's proof of \cite[Theorem 4.1]{Pe09}.

\begin{theorem}\label{thm.uniform-abelian}
Let $p \in M$ be a projection and $B \subset pMp$ an abelian von Neumann subalgebra that is normalized by a net of unitaries $(v_i)_{i \in I}$ in $\cU(pMp)$. Let $r \in pMp$ be any projection and make the following assumptions.
\begin{itemize}
\item If $t \recht 0$ then $\al_t \recht \id$ in \two uniformly on the tail of $(v_i)_{i \in I}$.
\item For every subset $\cF \subset \Gamma$ that is small relative to $\cS$, we have $\lim_i \|P_\cF(v_i r)\|_2 = 0$.
\end{itemize}
Denote by $Q$ the normalizer of $B$ inside $pMp$ and define $q$ as the smallest projection in $\cZ(Q)$ that satisfies $r \leq q$. Then $\al_t \recht \id$ in \two uniformly on $(Bq)_1$.
\end{theorem}

To prove this theorem we need two technical lemmas.

\begin{lemma}\label{lemma2}
Assume that $(v_i)_{i \in I}$ and $(w_i)_{i \in I}$ are bounded nets in $M$ such that $\|P_\cF(v_i)\|_2 \recht 0$ and $\|P_\cF(w_i)\|_2 \recht 0$ for all subsets $\cF \subset \Gamma$ that are small relative to $\cS$.
Then, $$\|E_M(x v_i y w_i z)\|_2 \recht 0 \quad\text{for all}\quad y \in \Mtil \ominus M , x,z \in \Mtil \; .$$
\end{lemma}
\begin{proof}
We can approximate $x,y,z$ by linear combinations of $(\om(\xi) \ot 1)a$, $\xi \in H_\R$, $a \in M$. If the nets $v_i,w_i$ satisfy the hypotheses of the lemma, the same is true, using Lemma \ref{lem.almost}, for the nets $a v_i b$ and $a w_i b$, given fixed elements $a,b \in M$. As a result it suffices to prove the lemma when
$$x = \om(\xi_1) \ot 1 \;\; , \quad y = (\om(\xi_2) - \exp(-\|\xi_2\|^2)1) \ot 1 \;\; , \quad z = \om(\xi_3) \ot 1$$
and $\xi_1,\xi_2,\xi_3 \in H_\R$. In that case a direct computation yields
$$E_M(x v_i y w_i z) = \exp(-\|\xi_1\|^2- \|\xi_2\|^2- \|\xi_3\|^2) \; \vphi_{31}( \vphi_{21}(v_i) \vphi_{32}(w_i) - v_i w_i)$$
where for all $\alpha,\beta \in \{1,2,3\}$ we define the completely bounded maps $\vphi_{\alpha\beta} : M \recht M$ given by
$$\vphi_{\alpha\beta}( a u_g) = \exp(-2 \langle \pi(g) \xi_\alpha, \xi_\beta\rangle) a u_g \quad\text{for all}\;\; a \in N, g \in \Gamma \; .$$
One checks that $\vphi_{\alpha\beta}(x) = \exp(\|\xi_\alpha\|^2+\|\xi_\beta\|^2) E_M(\om(\xi_\alpha) x \om(\xi_\beta))$ for all $x \in M$, implying that the $\vphi_{\alpha\beta}$ are indeed well defined completely bounded maps.

By our assumptions on $v_i$ and $w_i$, we have for all $\alpha,\beta\in\{1,2,3\}$ that $\|\vphi_{\alpha\beta}(v_i) - v_i \|_2 \recht 0$. Since $(\vphi_{\alpha\beta}(v_i))_i$ and $(\vphi_{\alpha\beta}(w_i))_i$ are moreover bounded nets, it follows that
$$\|\vphi_{21}(v_i) \vphi_{32}(w_i) - v_i w_i\|_2 \recht 0 \; .$$
But then also $\|E_M(x v_i y w_i z)\|_2 \recht 0$.
\end{proof}

In order to prove \ref{lemma2} we only used that $\pi$ is mixing relative to the family $\cS$ of subgroups of $\Gamma$. We now also use that the $1$-cocycle $b$ is bounded on every $\Sigma \in \cS$.

\begin{lemma}\label{lemma3}
Let $v_i$ and $w_i$ be bounded nets in $M$ such that $\|P_\cF(v_i)\|_2 \recht 0$ and $\|P_\cF(w_i)\|_2 \recht 0$ for all subsets $\cF \subset \Gamma$ that are small relative to $\cS$.

Let $x \in \Mtil \ominus M$, $t > 0$ and $\cF \subset \Gamma$ a subset that is small relative to $\cS$. Then,
$$\langle v_i x w_i , \al_t(P_\cF(\xi)) \rangle \recht 0 \quad\text{uniformly in}\quad \xi \in \rL^2(M), \|\xi\|_2 \leq 1 \; .$$
\end{lemma}
\begin{proof}
Fix $x \in \Mtil \ominus M$, $t > 0$ and a subset $\cF \subset \Gamma$ that is small relative to $\cS$. Write $\cF = \bigcup_{k=1}^m g_k \Sigma_k h_k$ with $\Sigma_k \in \cS$. We claim that there exist unitary elements $V_k, W_k \in \Mtil$ such that
$$| \langle v x w , \al_t(P_\cF(\xi)) \rangle | \leq \sum_{k=1}^m \|E_M(V_k v x w W_k)\|_2$$
for all $v,w \in M$ and all $\xi \in \rL^2(M)$ with $\|\xi\|_2 \leq 1$.
Once this claim is proven, the lemma follows from Lemma \ref{lemma2}.

Since the $1$-cocycle $b$ is bounded on $\Sigma_k$, we can take $\eta_k \in H_\R$ such that $t b(g) = \pi(g)\eta_k - \eta_k$ for all $g \in \Sigma_k$. Note that $\al_t(x) = \om(\eta_k)^* x \om(\eta_k)$ for all $x \in \rL^2(N \rtimes \Sigma_k)$. Put $V_k = \om(\eta_k) \al_t(u_{g_k}^*)$ and $W_k = \om(\eta_k)^* \al_t(u_{h_k}^*)$.

Take $\xi \in \rL^2(M), \|\xi\|_2 \leq 1$ arbitrarily. Let $\cF_k \subset \Sigma_k$ be such that $\bigcup_{k=1}^m g_k \Sigma_k h_k = \bigsqcup_{k=1}^m g_k \cF_k h_k$. Define $\xi_k = P_{\cF_k}(u_{g_k}^* \xi u_{h_k}^*)$. Note that $\xi_k \in \rL^2(N \rtimes \Sigma_k)$, that $\|\xi_k\|_2 \leq 1$ and that we have an orthogonal decomposition
$$P_\cF(\xi) = \sum_{k=1}^n u_{g_k} \xi_k u_{h_k} \; .$$
For all $k = 1,\ldots,m$, we have
$$|\langle v x w , \al_t(u_{g_k} \xi_k u_{h_k}) \rangle = |\langle E_M(V_k v x w W_k),\xi_k\rangle| \leq \| E_M(V_k v x w W_k) \|_2 \; .$$
Summing over $k$ yields the claim and hence proves the lemma.
\end{proof}

\subsection*{Proof of Theorem \ref{thm.uniform-abelian}}

Denote by $Q$ the normalizer of $B$ inside $pMp$. Denote by $q_1 \in \cZ(Q)$ the maximal projection given by Lemma \ref{lem.maxuniform-general} such that $\al_t \recht \id$ in \two uniformly on the unit ball of $Bq_1$. If $r \leq q_1$, then also $q \leq q_1$ and we are done. So assume that $r \not\leq q_1$.

Put $T := (p-q_1) r (p-q_1)$ and note that $T$ is nonzero. Let $r_0$ be a nonzero spectral projection of $T$ of the form $r_0 = T S$ for some $S \in M$. Since $v_i r_0 = (p-q_1) \; v_i r \; (p-q_1) S$, it follows from Lemma \ref{lem.almost} that $\|P_\cF(v_i r_0)\|_2 \recht 0$ for every subset $\cF \subset \Gamma$ that is small relative to $\cS$. Replace $p$ by $p-q_1$, $v_i$ by $v_i (p-q_1)$, $r$ by $r_0$ and $B$ by $B(p-q_1)$. We are now in a situation where $B \subset pMp$ is an abelian von Neumann subalgebra normalized by a net of unitaries $(v_i)_{i \in I}$ in $\cU(pMp)$ and where $r \in pMp$ is a nonzero projection such that the following properties hold.
\begin{itemize}
\item If $t \recht 0$ then $\al_t \recht \id$ in \two uniformly on the tail of $(v_i)_{i \in I}$.
\item For every subset $\cF \subset \Gamma$ that is small relative to $\cS$, we have $\lim_i \|P_\cF(v_i r)\|_2 = 0$.
\item By Proposition \ref{prop.maxuniform} there exists a sequence of unitaries $w_n \in \cU(B)$ such that for every $t > 0$ we have that $\|E_M(\al_t(w_n))\|_2 \recht 0$ as $n \recht \infty$.
\end{itemize}
We shall derive a contradiction from this list of three properties. We separately consider two cases.

{\bf Case 1.} For every $\eps > 0$ and every $b \in \cU(B)$ there exists a subset $\cF \subset \Gamma$ that is small relative to $\cS$ such that
$$
\liminf_i \|(1-P_\cF)(v_i b v_i^*)\|_2 < \eps \; .
$$

{\bf Case 2.} There exists a $\delta > 0$ and a unitary $b \in \cU(B)$ such that for every subset $\cF \subset \Gamma$ that is small relative to $\cS$, we have
\begin{equation}\label{eq.case2}
\limsup_i \|P_\cF(v_i b v_i^*)\|_2 \leq (1-\delta) \|p\|_2 \; .
\end{equation}

First assume that we are in case 1. Denote $\delta_t(b) = \al_t(b) - E_M(\al_t(b))$. We claim that $\delta_t \recht 0$ in \two uniformly on $r \cU(B) r$. To prove this statement, choose $\eps > 0$. Put $\delta = \eps^2/(9 \tau(p))$ and take $t > 0$ small enough and $i_0 \in I$ such that
$$\|r - \al_t(r)\|_2 \leq \delta \|p\|_2 \quad\text{and}\quad \|v_i - \al_t(v_i)\|_2 \leq \delta \|p\|_2 \quad\text{for all}\;\; i \geq i_0 \; .$$
We show that $\|\delta_t(rbr)\|_2 \leq \eps$ for all $b \in \cU(B)$, hence proving the claim above.

Note that our choice of $t$ implies that for all $b \in \cU(B)$
\begin{align*}
& \|r \al_t(b) r - \al_t(rbr)\|_2 \leq 2 \delta \|p\|_2 \quad\text{so that}\quad \|r \delta_t(b) r - \delta_t(rbr)\|_2 \leq 4 \delta \|p\|_2 \quad\text{and}\\
& \|v_i \al_t(b) v_i^* - \al_t(v_i b v_i^*)\|_2 \leq 2 \delta\|p\|_2 \quad\text{for all}\;\; i \geq I_0 \; .
\end{align*}
Fix $b \in \cU(B)$. Take a subset $\cF \subset \Gamma$ that is small relative to $\cS$ and such that $$\liminf_i \|(1-P_\cF)(v_i b v_i^*)\|_2 < \delta \|p\|_2 \; .$$
It follows that for all $i \geq i_0$,
\begin{align*}
\|\delta_t(rbr)\|_2^2 & = \langle \delta_t(rbr),\al_t(rbr)\rangle \leq | \langle r \delta_t(b) r , r \al_t(b) r \rangle | + 6 \delta \tau(p) \\
& = | \langle r \delta_t(b) r , \al_t(b) \rangle | + 6 \delta \tau(p) = | \langle v_i r \delta_t(b) r v_i^* , v_i \al_t(b) v_i^* \rangle | + 6 \delta \tau(p) \\
& \leq | \langle v_i r \delta_t(b) r v_i^* , \al_t(v_i b v_i^*) \rangle | + 8 \delta \tau(p) \\
& \leq | \langle v_i r \delta_t(b) r v_i^* , \al_t(P_\cF(v_i b v_i^*)) \rangle | + \|p\|_2 \; \|(1-P_\cF)(v_i b v_i^*)\|_2 + 8 \delta \tau(p) \; .
\end{align*}
Taking the $\liminf$ and using Lemma \ref{lemma3} it follows that $\|\delta_t(rbr)\|_2^2 \leq 9 \delta \tau(p) = \eps^2$, hence proving the claim above.

Denote by $q$ the smallest projection in $\cZ(Q)$ that satisfies $r \leq q$. We conclude from the claim above and from Proposition \ref{prop.extenduniform} that $\al_t \recht \id$ uniformly on the unit ball of $B q$. This is a contradiction with the existence of the sequence $w_n \in \cU(B)$ such that $\|E_M(\al_t(w_n))\|_2 \recht 0$ for every $t > 0$.

Next assume that we are in case 2. Take $\delta > 0$ and $b \in \cU(B)$ such that \eqref{eq.case2} holds. Write $\eps = \delta / 5$ and put $b_i := v_i b v_i^*$. Note that $\al_t \recht \id$ in \two uniformly on the tail of $(b_i)_{i \in I}$. Take $t > 0$ and $i_0 \in I$ such that $\|b_i - \al_t(b_i)\|_2 \leq \eps \|p\|_2$ for all $i \geq i_0$. We claim that $\|\delta_t(d)\|_2^2 \leq (1-\eps)\tau(p)$ for all $d \in \cU(B)$. To prove this claim fix $d \in \cU(B)$. Take a finite-dimensional subspace $K \subset D \ominus \C 1$ such that, using the notation of Lemma \ref{lemma1}, we have $\|(1 - Q_K)\delta_t(d)\|_2 \leq \eps \|p\|_2$. By Lemma \ref{lemma1} take a subset $\cF \subset \Gamma$ that is small relative to $\cS$ and such that $\|Q_K(b \al_t(d))\|_2 \leq \|P_\cF(b)\|_2 + \eps \|p\|_2$ for all $b$ in the unit ball of $M$.

By our choice of $t$ we know that $\|b_i \al_t(d) b_i^* - \al_t(b_i d b_i^*)\|_2 \leq 2 \eps \|p\|_2$ for all $i \geq i_0$. Since $B$ is abelian, also $d = b_i d b_i^*$.
Using that $Q_K$ is right $M$-modular it follows that for all $i \geq i_0$ we have
\begin{align*}
\|\delta_t(d)\|_2^2 &= \langle \al_t(d) , \delta_t(d) \rangle = \langle \al_t(b_i d b_i^*),\delta_t(d)\rangle
\leq |\langle b_i \al_t(d) b_i^*, \delta_t(d) \rangle | + 2 \eps \tau(p) \\
& \leq |\langle b_i \al_t(d) b_i^*, Q_K(\delta_t(d)) \rangle | + 3 \eps \tau(p) = |\langle Q_K(b_i \al_t(d) b_i^*), \delta_t(d) \rangle | + 3 \eps \tau(p) \\
& = |\langle Q_K(b_i \al_t(d)) \; b_i^*, \delta_t(d) \rangle | + 3 \eps \tau(p) \leq \|Q_K(b_i \al_t(d))\|_2 \; \|p\|_2 + 3 \eps \tau(p) \\
& \leq \|P_\cF(b_i)\|_2 \; \|p\|_2 + 4 \eps \tau(p) \; .
\end{align*}
Taking the $\limsup$ it follows that $\|\delta_t(d)\|_2^2 \leq (1-\delta+4\eps) \tau(p) = (1-\eps) \tau(p)$, hence proving the claim.

From this claim, it follows that $\|E_M(\al_t(d))\|_2^2 \geq \eps \tau(p)$ for all $d \in \cU(B)$. This is a contradiction with the existence of the sequence $(w_n)$ in $\cU(B)$ such that $\|E_M(\al_t(w_n))\|_2 \recht 0$. This ends the proof of case 2 and also ends the proof of Theorem \ref{thm.uniform-abelian}.\hfill\qedsymbol

\section{Transfer of rigidity}

We fix a trace preserving action $\Gamma \actson (N,\tau)$ and put $M = N \rtimes \Gamma$. Let $\cS$ be a family of subgroups of $\Gamma$. As above we call a subset $\cF \subset \Gamma$ small relative to $\cS$ if $\cF$ can be written as a finite union of subsets of the form $g \Sigma h$, $g, h \in \Gamma$, $\Sigma \in \cS$. Whenever $\cF \subset \Gamma$ we denote by $P_\cF$ the orthogonal projection of $\rL^2(M)$ onto the closed linear span of $\{a u_g \mid a \in N, g \in \cF\}$.

Let $f:\Gamma \recht \R$ be a conditionally negative type function with $f(e)=0$. Define the semigroup $(\vphi_t)_{t > 0}$ of unital trace preserving completely positive maps
$$\vphi_t : M \recht M : \vphi_t(a u_g) = \exp(-t f(g)) a u_g \quad\text{for all}\;\; a \in N , g \in \Gamma \; .$$

\begin{proposition}\label{prop.transfer}
Let $p \in M$ be a projection and assume that $pMp = B \rtimes \Lambda$ is any crossed product decomposition with corresponding canonical unitaries $(v_s)_{s \in \Lambda}$. Let $\Delta : pMp \recht pMp \ovt \rL\Lambda$ be the comultiplication given by $\Delta(b v_s) = b v_s \ot v_s$ for all $b \in B$, $s \in \Lambda$. Assume that $(w_i)_{i \in I}$ is a net of unitaries in $\cU(pMp)$ and that $q \in (\rL \Lambda)' \cap pMp$ is a projection satisfying
\begin{itemize}
\item if $t \recht 0$ then $\id \ot \vphi_t \recht \id$ in \two uniformly on the tail of $(\Delta(w_i))_{i \in I}$,
\item for every subset $\cF \subset \Gamma$ that is small relative to $\cS$, we have $\lim_i \|(1 \ot P_\cF)(\Delta(w_i)(1 \ot q))\|_2 = 0$.
\end{itemize}
Then there exists a net of elements $(s_j)_{j \in J}$ in $\Lambda$ such that, writing $v_j := v_{s_j}$, the following holds.
\begin{itemize}
\item If $t \recht 0$ then $\vphi_t \recht \id$ in \two uniformly on the tail of $(v_j)_{j \in J}$.
\item For every subset $\cF \subset \Gamma$ that is small relative to $\cS$, we have $\lim_j \|P_\cF(v_j q)\|_2 = 0$.
\end{itemize}
\end{proposition}

\begin{proof}
Normalize the trace on $M$ such that $\tau(p) = 1$. As such $\Delta$ is trace preserving. Take a decreasing sequence $t_1 > t_2 > \cdots$ of strictly positive numbers and an increasing sequence $i_1 \leq i_2 \leq \cdots$ such that
$$1 - \Re (\tau \ot \tau)\bigl(\Delta(w_i)^*(\id \ot \vphi_{t_n})\Delta(w_i)\bigr) \leq 4^{-n-1} \quad\text{for all} \;\; i \geq i_n \; .$$
Define
$$\cV_n := \{s \in \Lambda \mid 1 - \Re \tau(v_s^* \vphi_{t_n}(v_s)) \leq 2^{-n-1} \} \; .$$
Fix $n \in \N$ and $i \geq i_n$. Write $w_i = \sum_{s \in \Lambda} w^i_s v_s$ with $w^i_s \in B$. It follows that
\begin{align*}
4^{-n-1} & \geq 1 - \Re (\tau \ot \tau)\bigl(\Delta(w_i)^*(\id \ot \vphi_{t_n})\Delta(w_i)\bigr) \\
& = \sum_{s \in \Lambda} \bigl( 1 - \Re \tau(v_s^* \vphi_{t_n}(v_s)) \bigr) \; \|w^i_s\|_2^2 \\
& \geq \sum_{s \in \Lambda - \cV_n} 2^{-n-1} \|w^i_s\|_2^2 \; .
\end{align*}
We conclude that for all $n \in \N$ and all $i \geq i_n$,
$$\sum_{s \in \Lambda - \cV_n} \|w^i_s\|_2^2 \leq 2^{-n-1} \; .$$
Define $\cW_n := \cV_1 \cap \cdots \cap \cV_n$. It follows that for all $i \geq i_n$
$$\sum_{s \in \Lambda-\cW_n} \|w^i_s\|_2^2 \leq \frac{1}{2} \quad\text{and hence}\quad   \sum_{s \in \cW_n} \|w^i_s\|_2^2 \geq \frac{1}{2} \; .$$
We claim that for every $\eps > 0$, $n \in \N$ and subset $\cF \subset \Gamma$ that is small relative to $\cS$, there exists an $s \in \cW_n$ satisfying $\|P_\cF(v_s q)\|_2 < \eps$. Indeed, if for a given $\eps > 0$, $n \in \N$ and $\cF \subset \Gamma$ that is small relative to $\cS$, the claim fails, it would follow that for all $i \geq i_n$,
\begin{align*}
\|(\id \ot P_\cF)(\Delta(w_i)(1 \ot q))\|_2^2 &= \sum_{s \in \Lambda} \|w^i_s\|_2^2 \; \|P_\cF(v_s q)\|_2^2 \\
&\geq \sum_{s \in \cW_n} \|w^i_s\|_2^2 \; \|P_\cF(v_s q)\|_2^2 \geq \sum_{s \in \cW_n} \|w^i_s\|_2^2 \; \eps^2 \geq \frac{\eps^2}{2} \; .
\end{align*}
Since $\lim_i \|(\id \ot P_\cF)(\Delta(w_i)(1 \ot q))\|_2 = 0$ this is absurd and the claim is proven.

For every $\eps > 0$, $n \in \N$ and $\cF \subset \Gamma$ small relative to $\cS$, pick an element $s_{\eps,n,\cF} \in \cW_n$ such that $\|P_\cF(v_{s_{\eps,n,\cF}} q)\|_2 < \eps$. We obtain a net $(s_j)_{j \in J}$ in $\Lambda$ that satisfies all the conclusions of the proposition. Indeed, first observe that for every fixed $s \in \Lambda$ the expression $\tau(v_s^* \vphi_t(v_s))$ increases when $t > 0$ decreases. Writing $v_{\eps,n,\cF} = v_{s_{\eps,n,\cF}}$ and using the inequality $\|\vphi_t(v_s) - v_s\|_2^2 \leq 2(1-\Re \tau(v_s^* \vphi_t(v_s)))$, it follows that
\begin{align*}
\|\vphi_t(v_{\eps,n,\cF}) - v_{\eps,n,\cF}\|_2 \leq 2^{-n_0/2} &\quad\text{whenever}\;\; 0 < t \leq t_{n_0} \;\;\text{and}\;\; n \geq n_0 \; ,\\
\|P_{\cF_0}(v_{\eps,n,\cF} \, q)\|_2 < \eps &\quad\text{whenever}\;\; \cF_0 \subset \cF \; .
\end{align*}
\end{proof}

\section{Proof of Theorems \ref{thm.main} and \ref{thm.uniquecartan}}

We are given a crossed product II$_1$ factor $M = N \rtimes \Gamma$. Fix a projection $p \in M$ and assume that $pMp = B \rtimes \Lambda$ is another crossed product decomposition with $B$ being diffuse and of type I. Denote by $(v_s)_{s \in \Lambda}$ the canonical unitaries in $B \rtimes \Lambda$. Since $M$ is a factor, the action $\Lambda \actson \cZ(B)$ is ergodic and hence $B \cong \M_m(\C) \ot \cZ(B)$ for some integer $m$.

Denote by $\Delta : pMp \recht pMp \ovt \rL \Lambda$ the comultiplication given by $\Delta(b v_s) = b v_s \ot v_s$ for all $b \in B$, $s \in \Lambda$.

Since $M$ is a II$_1$ factor and $B \subset pMp$ is diffuse, we can take partial isometries $V_1,\ldots,V_k \in M$ such that $V_1 = p$,  $V_i^* V_i \in B$ for all $i=1,\ldots,k$ and $\sum_i V_i V_i^* = 1$. We extend $\Delta$ to a unital $*$-homomorphism $M \recht M \ovt \rL \Lambda$ by the formula
$$\Delta(x) := \sum_{i,j=1}^k (V_i \ot 1) \Delta(V_i^* x V_j) (V_j^* \ot 1) \; .$$
Since $V_1 = p$, the restriction of the new $\Delta$ to $pMp$ equals the original comultiplication.

The following meta-theorem brings together all that we have done in the previous sections. Our main Theorem \ref{thm.main} will be a direct consequence.

\begin{theorem}\label{thm.meta}
Within the setup described before the theorem, let $b$ be an unbounded $1$-cocycle into the orthogonal representation $\pi : \Gamma \recht \cO(H_\R)$ that is mixing relative to a family $\cS$ of subgroups of $\Gamma$ such that $b$ is bounded on every $\Sigma \in \cS$. Define $\Mtil$ and $(\al_t)_{t \in \R}$ as in paragraph \ref{subsec.deform}. Assume that $Q \subset M$ is a diffuse von Neumann subalgebra and that $q \in (\rL \Lambda)' \cap pMp$ is a nonzero projection such that the following two conditions hold.
\begin{enumerate}
\item $\id \ot \al_t \recht \id$ in \two uniformly on the unit ball of $\Delta(Q)$.
\item For every $\Sigma \in \cS$ we have $\Delta(Q)(1 \ot q) \not\prec M \ovt (N \rtimes \Sigma)$.
\end{enumerate}
Then there exists $\Sigma \in \cS$ such that $B \prec N \rtimes \Sigma$ and hence $(B)_1 \almost N \rtimes \cS$.
\end{theorem}
\begin{proof}
Since $Q$ is diffuse, we may, after a unitary conjugacy of $Q$, assume that $p \in Q$. It follows that $\Delta(pQp)(1 \ot q) \not\prec pMp \ovt (N \rtimes \Sigma)$ for all $\Sigma \in \cS$. Lemma \ref{lem.intertwining} provides a net of unitaries $(w_i)$ in $\cU(pQp)$ such that $\|(1 \ot P_\cF)(\Delta(w_i)(1 \ot q))\|_2 \recht 0$ for every subset $\cF \subset \Gamma$ that is small relative to $\cS$. Since $\id \ot \al_t \recht \id$ in \two uniformly on the unit ball of $\Delta(Q)$, certainly $\id \ot \al_t \recht \id$ in \two uniformly on the unitaries $\Delta(w_i) \in \cU(pMp \ovt pMp)$. By the transfer of rigidity proposition \ref{prop.transfer} we find a net of elements $(s_j)$ in $\Lambda$ such that, writing $v_j := v_{s_j}$, we have that $\al_t \recht \id$ in \two uniformly on the tail of $(v_j)$ and $\|P_\cF(v_j q)\|_2 \recht 0$ for every subset $\cF \subset \Gamma$ that is small relative to $\cS$.

Since $\cZ(B)$ is an abelian von Neumann subalgebra of $pMp$ that is normalized by the unitaries $v_j$ and that is moreover regular in $pMp$, it follows from Theorem \ref{thm.uniform-abelian} and Lemma \ref{lem.maxuniform-general} that $\al_t \recht \id$ in \two uniformly on the unit ball of $\cZ(B)$, and hence as well on the unit ball of $B = \M_m(\C) \ot \cZ(B)$. If for every $\Sigma \in \cS$ we would have that $B \not\prec N \rtimes \Sigma$, Theorem \ref{thm.uniform-normalizer} would imply that $\al_t \recht \id$ in \two uniformly on the unit ball of $pMp$. This would be a contradiction with $b$ being unbounded.

So there exists a $\Sigma \in \cS$ such that $B \prec N \rtimes \Sigma$. Since $B \subset pMp$ is regular Proposition \ref{prop.maxintertwine} implies that $(B)_1 \almost N\rtimes \cS$.
\end{proof}

In order to establish the condition $\Delta(Q)(1 \ot q) \not\prec M \ovt (N \rtimes \Sigma)$ appearing in Theorem \ref{thm.meta}, we prove the following lemma. It is contained in \cite[Lemma 9.2]{IPV10} and \cite[Lemma 4]{HPV10} but we include a proof for the convenience of the reader.

We say that a finite von Neumann algebra $P$ is anti-(T) if there exists a chain of von Neumann subalgebras $\C1 = P_0 \subset P_1 \subset \cdots \subset P_n = P$ such that for every $i = 1,\ldots,n$ the finite von Neumann algebra $P_i$ has property (H) relative to $P_{i-1}$ in the sense of \cite[Section 2]{Po01}. Examples include crossed products $A \rtimes \Sigma$ where $A$ is amenable and $\Sigma$ admits a chain of subgroups $\{e\} = \Sigma_0 < \Sigma_1 < \cdots < \Sigma_n = \Sigma$ such that for all $i \in \{1,\ldots,n\}$ the subgroup $\Sigma_{i-1} \lhd \Sigma_i$ is normal and the quotient group $\Sigma_i / \Sigma_{i-1}$ has the Haagerup property.

\begin{lemma}\label{lem.no-intertwine}
If the bimodule $\bim{pMp}{\cH}{pMp}$ is weakly contained in the coarse $pMp$-$pMp$-bimodule then the bimodule $\bim{\Delta(M)}{(\rL^2(M) \ot \cH)}{\Delta(M)}$ is weakly contained in the coarse $M$-$M$-bimodule.

Let $Q,P \subset M$ be von Neumann subalgebras.
\begin{itemize}
\item If $Q$ has no amenable direct summand and $P$ is amenable, then $\Delta(Q) \not\prec M \ovt P$.
\item If $Q$ is diffuse with property (T) and $P$ is anti-(T), then $\Delta(Q) \not\prec M \ovt P$.
\end{itemize}
\end{lemma}
\begin{proof}
Put $\cM = B \rtimes \Lambda = pMp$. Denote by $\si : \cM \ovt \cM \recht \cM \ovt \cM$ the flip automorphism. We first claim that the bimodule
$$\bim{\Delta(\cM) \ot 1}{\rL^2(\cM \ovt \cM \ovt \cM)}{(\id \ot \si)(\Delta(\cM) \ot 1)}$$
is weakly contained in the coarse $\cM$-$\cM$-bimodule. To prove this claim, observe that
$$\bim{\Delta(\cM) \ot 1}{\bigl(\rL^2(\cM) \ot \ell^2(\Lambda) \ot \rL^2(\cM)\bigr)}{\cM \ovt 1 \ovt \cM}$$
is unitarily isomorphic with the tensor product $\rL^2(\cM) \ot_B \rL^2(\cM \ovt \cM)$ of the bimodules $\bim{\cM}{\rL^2(\cM)}{B}$ and $\bim{B \ot 1}{\rL^2(\cM \ovt \cM)}{\cM \ovt \cM}$.
Since $B$ is amenable this relative tensor product is weakly contained in the coarse bimodule. Restricting the right $(\cM \ovt \cM)$-module action to $\Delta(\cM)$, the claim follows. We extended $\Delta$ from $\cM$ to $M$. In the bimodule picture this amounts to tensoring on the left by $\rL^2(Mp)$ and on the right by $\rL^2(pM)$. It follows that
$$\bim{\Delta(M) \ot 1}{\rL^2(M \ovt \cM \ovt \cM)}{(\id \ot \si)(\Delta(M) \ot 1)}$$
is weakly contained in the coarse $M$-$M$-bimodule. From this the first statement of the lemma follows immediately.

Assume that $P$ is amenable and that $\Delta(Q) \prec M \ovt P$. We prove that $Q$ has an amenable direct summand.
Since $P$ is amenable, the coarse $P$-$P$-bimodule $\bim{P \ot 1}{\rL^2(P \ovt P)}{1 \ot P}$ contains a sequence of vectors $\xi_n$ satisfying the following properties.
$$\|(a \ot 1) \xi_n - \xi_n(1 \ot a)\|_2 \recht 0 \quad\text{and}\quad \langle (a \ot 1) \xi_n,\xi_n \rangle \recht \tau(a) \quad\text{for all}\;\; a \in P \; .$$
View $\rL^2(P \ovt P) \subset \rL^2(M \ovt M)$ and identify $\rL^2(M \ovt M)$ with the space of Hilbert Schmidt operators on $\rL^2(M)$.
Since $\xi_n \in \rL^2(P \ovt P)$ we have for all $a \in M$ that
\begin{align*}
\langle (a \ot 1)\xi_n,\xi_n \rangle &= (\tau \ot \tau)((a \ot 1) \xi_n \xi_n^*) = (\tau \ot \tau)((E_P(a) \ot 1)\xi_n \xi_n^*) \\
&= \langle (E_P(a) \ot 1)\xi_n,\xi_n \rangle \recht \tau(E_P(a)) = \tau(a) \; .
\end{align*}
So every $\xi_n$ gives rise to a Hilbert Schmidt operator $S_n$ on $\rL^2(M)$ and hence a trace class operator $T_n := S_n S_n^* \in \TC(\rL^2(M))^+$. By construction, $T_n$ has the following properties
$$\|a T_n - T_n a\|_{1,\Tr} \recht 0 \;\;\text{for all}\;\; a \in P \quad\text{and}\quad \Tr(b T_n) \recht \tau(b) \;\;\text{for all}\;\; b \in M \; .$$
Since $\Delta(Q) \prec M \ovt P$, take a nonzero partial isometry $v \in \M_{1,k}(\C) \ovt M \ovt pM$ and a, possibly non-unital, $*$-homomorphism $\theta : Q \recht \M_k(\C) \ovt M \ovt P$ satisfying $\Delta(a) v = v \theta(a)$ for all $a \in Q$. Denote $q := vv^*$ and note that $q \in \Delta(Q)' \cap M \ovt \cM$.

The operator $R_n:=v (1 \ot 1 \ot T_n) v^*$ is a positive element of $M \ovt \B(p \rL^2(M))$ satisfying $(\tau \ot \Tr)(R_n) < \infty$. The square root $\eta_n := R_n^{1/2}$ can be viewed as a vector in $\rL^2(M) \ot p \rL^2(M) \ot \rL^2(M)p$ and satisfies by construction the following properties.
\begin{align*}
& \|(\Delta(a) \ot 1) \eta_n - \eta_n (\id \ot \si)(\Delta(a) \ot 1)\|_2 \recht 0 \;\;\text{for all}\;\; a \in Q \quad\text{and} \\
& \langle (b \ot 1)\eta_n,\eta_n \rangle \recht \tau(bq) \;\;\text{for all}\;\; b \in M \ovt \cM \; .
\end{align*}
Define $z \in \cZ(Q)$ such that $\Delta(z)$ is the support projection of $E_{\Delta(Q)}(q)$. We have shown that the bimodule
\begin{equation}\label{eq.mybimodule}
\bim{\Delta(Q) \ot 1}{\rL^2(M \ovt \cM \ovt \cM)}{(\id \ot \si)(\Delta(Q) \ot 1)}
\end{equation}
weakly contains the trivial $Qz$-bimodule. By the first statement of the lemma the bimodule in \eqref{eq.mybimodule} is weakly contained in the coarse $Q$-$Q$-bimodule. It follows that $Qz$ is an amenable direct summand of $Q$.

Finally assume that $P$ is anti-(T), that $Q$ is diffuse with property (T) and that $\Delta(Q) \prec M \ovt P$. Let $\C 1 = P_0 \subset P_1 \subset \cdots \subset P_n = P$ be a chain of von Neumann subalgebras such that for every $i = 1,\ldots,n$, $P_i$ has property (H) relative to $P_{i-1}$. Repeatedly applying \cite[Lemma 1]{HPV10} it follows that $\Delta(Q) \prec M \ovt P_i$ for every $i$ and hence $\Delta(Q) \prec M \ovt 1$. By the previous statement of the lemma, $Q$ has an amenable direct summand. This is a contradiction with $Q$ being diffuse with property (T).
\end{proof}

\subsection*{Proof of Theorem \ref{thm.main}}

We finally prove Theorem \ref{thm.main}. We now also require that $N = A$ is of type I. Since $M$ is a factor the action $\Gamma \actson \cZ(A)$ is ergodic and since $A$ is of type I it follows that $A \cong \M_n(\C) \ot \cZ(A)$ for some integer $n$.

{\bf\boldmath The case $\Gamma \in \cC$.} Let $H < \Gamma$ be a nonamenable subgroup with the relative property (T). Then, $Q = \rL H$ satisfies the conditions in Theorem \ref{thm.meta}. Condition 1 follows directly from the relative property (T) of $\Delta(Q)$ inside $M \ovt pMp$. Condition 2 follows from Lemma \ref{lem.no-intertwine} and the observation that all the von Neumann algebras $A \rtimes \Sigma$, $\Sigma \in \cS$, are amenable.

{\bf\boldmath The case $\Gamma \in \cD$.} Let $H < \Gamma$ be an infinite subgroup with the plain property (T). Then, $Q = \rL H$ satisfies the conditions in Theorem \ref{thm.meta}. Condition 1 follows directly from the property (T) of $\Delta(Q)$, while condition 2 follows from Lemma \ref{lem.no-intertwine} and the observation that all the von Neumann algebras $A \rtimes \Sigma$, $\Sigma \in \cS$, are anti-(T).

{\bf\boldmath The case $\Gamma \in \cE$.} Let $H < \Gamma$ be a nonamenable subgroup with a nonamenable centralizer $H' < \Gamma$. We claim that $Q = \rL H$ satisfies the conditions of Theorem \ref{thm.meta}. Condition 2 follows from Lemma \ref{lem.no-intertwine} and the observation that all the von Neumann algebras $A \rtimes \Sigma$, $\Sigma \in \cS$, are amenable. We now prove condition 1 using a spectral gap argument.

We are given the orthogonal representation $\pi : \Gamma \recht \cO(H_\R)$ that is weakly contained in the regular representation and the $1$-cocycle $b : \Gamma \recht H_\R$. With these data we build the automorphisms $(\al_t)$ on $\Mtil$ as in paragraph \ref{subsec.deform}. Define $\cH := p \rL^2(\Mtil \ominus M) p$. From Lemma \ref{lem.regular} it follows that $\bim{pMp}{\cH}{pMp}$ is weakly contained in the coarse $pMp$-$pMp$-bimodule. By Lemma \ref{lem.no-intertwine} the bimodule $\bim{\Delta(M)}{(\rL^2(M) \ot \cH)}{\Delta(M)}$ is weakly contained in the coarse $M$-$M$-bimodule. Therefore the unitary representation
$$\gamma : H' \recht \cU(\rL^2(M) \ot \cH) : \gamma(g) \xi := \Delta(u_g) \xi \Delta(u_g^*)$$
is weakly contained in the regular representation. Since $H'$ is nonamenable, $\gamma$ does not weakly contain the trivial $H'$-representation.

Choose $\eps > 0$. Take $g_1,\ldots,g_n \in H'$ and $\rho > 0$ such that every vector $\xi \in \rL^2(M) \ot \cH$ satisfying $\|\xi - \gamma(g_k)\xi\| \leq \rho$ for all $k = 1,\ldots,n$, must be of norm at most $\eps$, i.e.\ satisfies $\|\xi\| \leq \eps$. Put $\delta = \min\{\rho/12,\eps\}$. Take $t > 0$ small enough such that
$$\|\al_t(p) - p\|_2 \leq \delta \quad\text{and}\quad \|(\id \ot \al_t)\Delta(u_{g_k}) - \Delta(u_{g_k})\|_2 \leq \delta \quad\text{for all}\;\; k=1,\ldots,n \; .$$
We claim that $\|(\id \ot \al_t)\Delta(b) - \Delta(b)\|_2 \leq 5 \sqrt{2} \eps$ for all $b$ in the unit ball of $Q$. Once this claim is proven, also condition 1 of Theorem \ref{thm.meta} has been verified and the theorem follows from Theorem \ref{thm.meta}.

Fix $b \in (Q)_1$. Write $\xi = (\id \ot \al_t)\Delta(b)$, $\xi' = (1 \ot p)\xi(1 \ot p)$ and $\xi^{\prime\prime} = \xi' - (\id \ot E_M)(\xi')$. Note that $\xi^{\prime\prime} \in \rL^2(M) \ot \cH$. Observe that $\|\xi-\xi'\|_2 \leq 2\delta$. Since $b$ commutes with $u_{g_k}$, we get that $\|\xi - \Delta(u_{g_k}) \xi \Delta(u_{g_k}^*)\|_2 \leq 2\delta$ for all $k = 1,\ldots,n$. Hence, $\|\xi' - \Delta(u_{g_k}) \xi' \Delta(u_{g_k}^*)\|_2 \leq 6\delta$ and so $\|\xi^{\prime\prime} - \gamma(g_k) \xi^{\prime\prime}\|_2 \leq 12\delta\leq\rho$. We conclude that $\|\xi^{\prime\prime}\|_2 \leq \eps$. Hence, $\|\xi - (\id \ot E_M)(\xi)\|_2 \leq \eps + 4\delta\leq 5\eps$. This means that $\|(\id \ot \delta_t)\Delta(b)\|_2 \leq 5\eps$. It follows from Lemma \ref{lem.equivalent} that $\|\Delta(b) - (\id \ot \al_t)\Delta(b)\|_2 \leq 5 \sqrt{2}\eps$, proving the claim above.

{\bf\boldmath The case $\Gamma \in \cC_2$.} We have $\Gamma = \Gamma_1 \times \Gamma_2$ with $\Gamma_i \in \cC$. So we are given nonamenable subgroups $H_i < \Gamma_i$ with the relative property (T) and families $\cS_i$ of subgroups of $\Gamma_i$. We can view $M$ as the crossed product $M = (A \rtimes \Gamma_1) \rtimes \Gamma_2$ or as the crossed product $(A \rtimes \Gamma_2)\rtimes \Gamma_1$. This gives rise to the malleable deformations $(\al^1_t)$ and $(\al^2_t)$ associated with the unbounded cocycles $b_i : \Gamma_i \recht H^i_\R$ into the orthogonal representations $\pi^i : \Gamma_i \recht \cO(H^i_\R)$.

Denote $Q = \rL(H_1 \times H_2)$. We prove below the existence of
a nonzero projection $q \in pMp \cap (\rL \Lambda)'$ such that for all $\Sigma_2 \in \cS_2$, we have
\begin{equation}\label{eq.versionclaim}
\Delta(Q)(1 \ot q) \not\prec pMp \ovt (A \rtimes (\Gamma_1 \times \Sigma_2))
\end{equation}
Since $\Delta(Q) \subset M \ovt pMp$ has the relative property (T), \eqref{eq.versionclaim} and Theorem \ref{thm.meta} imply that $B \almost (A \rtimes \Gamma_1) \rtimes \cS_2$. By symmetry also $B \almost (A \rtimes \Gamma_2) \rtimes \cS_1$. By Lemma \ref{lem.intersection} it follows that there exist $\Sigma_i \in \cS_i$ such that $B \prec A \rtimes (\Sigma_1 \times \Sigma_2)$, ending the proof of the case $\Gamma \in \cC_2$. It remains to settle \eqref{eq.versionclaim}.

Denote by $\ball N$ the unit ball of a von Neumann algebra $N$. Let $p_1 \in \Delta(\rL(H_1 \times H_2))' \cap M \ovt pMp$ be as in Proposition \ref{prop.maxintertwine} the maximal projection such that $$\ball\bigl(\Delta(\rL(H_1 \times H_2))p_1) \almost M \ovt (A \rtimes (\Gamma_1 \times \cS_2)) \; .$$ Define $f_i$ as the smallest projection in $\Delta(\rL(\Gamma_i))' \cap M \ovt pMp$ that satisfies $p_1 \leq f_i$. Define $p_2$ as the smallest projection in $\Delta(\rL \Gamma)' \cap M \ovt pMp$ that satisfies $p_1 \leq p_2$. Note that $f_i \leq p_2$ for both $i=1,2$.

Since $\ball\bigl(\Delta(\rL(H_2)) p_1) \almost M \ovt (A \rtimes (\Gamma_1 \times \cS_2))$ and since $H_2$ is nonamenable, Lemmas \ref{lem.intersection} and \ref{lem.no-intertwine} imply that $\Delta(\rL(H_2))p_1 \not\prec M \ovt (A \rtimes (\cS_1 \times \Gamma_2))$. Since $H_2 < \Gamma_2$ has the relative property (T), we know that $\id \ot \al^1_t \recht \id$ in \two uniformly on the unit ball of $\Delta(\rL(H_2))p_1$. Both statements, together with Theorem \ref{thm.uniform-normalizer} and the observation that $\rL(\Gamma_1)$ commutes with $\rL(H_2)$, imply that $\id \ot \al^1_t \recht \id$ in \two uniformly on the unit ball of $\Delta(\rL(\Gamma_1))f_1$. Since the normalizer of $\rL(\Gamma_1)$ contains $\rL(\Gamma)$, Lemma \ref{lem.maxuniform-general} implies that $\id \ot \al^1_t \recht \id$ in \two uniformly on the unit ball of $\Delta(\rL(\Gamma_1))p_2$. Using $H_1$ instead of $H_2$, we also find that $\id \ot \al^1_t \recht \id$ in \two uniformly on the unit ball of $\Delta(\rL(\Gamma_2))p_2$. The unitaries $\Delta(u_{(g,h)}) = \Delta(u_{(g,e)}) \Delta(u_{(e,h)})$ form a group generating $\Delta(\rL \Gamma)$. It follows from Proposition \ref{prop.extenduniform} that $\id \ot \al^1_t \recht \id$ in \two uniformly on the unit ball of $\Delta(\rL \Gamma) p_2$.

Define $p_3$ as the smallest projection in $\Delta(M)' \cap M \ovt pMp$ that satisfies $p_1 \leq p_3$. Note that $p_2 \leq p_3$. We observed above that
$\Delta(\rL(H_2))p_1 \not\prec M \ovt (A \rtimes (\cS_1 \times \Gamma_2))$. The relative property (T) of $\Delta(\rL(H_2))$ in $M \ovt pMp$ implies that
 $\id \ot \al^1_t \recht \id$ in \two uniformly on the unit ball of $\Delta(\rL(H_2))$. By Lemma \ref{lem.intertwining} take a net of elements $g_i \in H_2$ such that $\|(1 \ot P_{\cF \times \Gamma_2})(\Delta(u_{g_i})p_1)\|_2 \recht 0$ for every subset $\cF \subset \Gamma_1$ that is small relative to $\cS_1$. Since the unitaries $\Delta(u_{g_i})$ normalize the abelian von Neumann algebra $\Delta(\cZ(A))$ and since $\cZ(A) \subset M$ is regular, it follows from Theorem \ref{thm.uniform-abelian} that $\id \ot \al^1_t \recht \id$ in \two uniformly on the unit ball of $\Delta(\cZ(A))p_3$ and hence also on the unit ball of $\Delta(A) p_3$. The group of unitaries $\Delta(au_g)$, $a \in \cU(A)$, $g \in \Gamma$, generates $\Delta(M)$. So, the uniform convergence $\id \ot \al^1_t \recht \id$ on the unit balls of $\Delta(\rL \Gamma) p_2$ and $\Delta(A) p_3$, together with Proposition \ref{prop.extenduniform}, implies that $\id \ot \al^1_t \recht \id$ in \two uniformly on the unit ball of $\Delta(M)p_3$.

Recall how $\Delta$ was extended to $M$ starting from the comultiplication on $pMp$. Put $S_i := V_i V_i^*$ and $P_i := V_i^* V_i \in B$. By construction $\Delta(S_i) = S_i \ot 1$ and hence $p_3$ commutes with all the projections $S_i \ot 1$. Define $p_i := (V_i^* \ot 1) p_3 (V_i \ot 1)$. For all $b \in \cU(B)$ and $s,r \in \Lambda$, we have
\begin{equation}\label{eq.tussenstap}
p_i (b v_s \ot v_r) = (V_i^* \ot 1) p_3 \Delta(V_i b v_r) (v_{r^{-1}s} \ot 1) \; .
\end{equation}
We know that $\id \ot \al^1_t \recht \id$ in \two uniformly on the unit ball of $p_3\Delta(M)$. So, $\id \ot \al^1_t \recht \id$ in \two uniformly on the elements $\Delta(V_i b v_r)$, $b \in \cU(B)$, $r \in \Lambda$. Formula \eqref{eq.tussenstap} then implies that $\id \ot \al^1_t \recht \id$ in \two uniformly on the elements $p_i (b v_s \ot v_r)$, $b \in \cU(B)$, $s,r \in \Lambda$.
As in Lemma \ref{lem.maxuniform-general} let $q_1 \in (\rL \Lambda)' \cap pMp$ be the maximal projection such that $\al^1_t \recht \id$ in \two uniformly on the unit ball of $\rL(\Lambda) q_1$.
Since the unitaries $b v_s \ot v_r$ form a group generating $pMp \ovt \rL \Lambda$ and since $\id \ot \al^1_t \recht \id$ in \two uniformly on the elements $p_i (b v_s \ot v_r)$, it follows from Proposition \ref{prop.extenduniform} that $p_i \leq p \ot q_1$ for every $i = 1,\ldots,k$. So $p_3 \leq 1 \ot q_1$ and in particular $p_1 \leq 1 \ot q_1$. If $q_1 < p$, we can put $q = p-q_1$ and \eqref{eq.versionclaim} is proven.

As a final step, we assume that $q_1 = p$ and derive a contradiction. So, $\al^1_t \recht \id$ in \two uniformly on the unit ball of $\rL(\Lambda)$. We observed above that
$\Delta(\rL(H_2))p_1 \not\prec M \ovt (A \rtimes (\cS_1 \times \Gamma_2))$. Since $\Delta(M) \subset M \ovt \rL \Lambda$, it is then impossible that
$$\ball\bigl(M \ovt \rL(\Lambda)) \almost M \ovt (A \rtimes (\cS_1 \times \Gamma_2)) \; .$$
By Lemma \ref{lem.intertwining} and Proposition \ref{prop.maxintertwine} we find a nonzero projection $e \in (\rL \Lambda)' \cap pMp$ and a net of elements $(s_j)$ in $\Lambda$ such that $\|P_{\cF \times \Gamma_2}(v_{s_j} e)\|_2 \recht 0$ for every subset $\cF \subset \Gamma_1$ that is small relative to $\cS_1$. Since the unitaries $v_{s_j}$ normalize the regular abelian von Neumann subalgebra $\cZ(B) \subset pMp$, it follows from Theorem \ref{thm.uniform-abelian} that $\al^1_t \recht \id$ in \two uniformly on the unit ball of $\cZ(B)$. Together with the uniform convergence on the unit ball of $\rL(\Lambda)$ and Proposition \ref{prop.extenduniform}, we obtain the uniform convergence on the unit ball of $pMp$. This is absurd because the cocycle $b_1$ is unbounded.

{\bf\boldmath The case $\Gamma \in \cD_2$.} The proof is identical to the proof of the case $\Gamma \in \cC_2$, but now using the last statement of Lemma \ref{lem.no-intertwine}.\hfill\qedsymbol

\subsection*{Proof of Theorem \ref{thm.uniquecartan}}

Both amalgamated free products and HNN extensions admit a natural action on their Bass-Serre tree, yielding $1$-cocycles into orthogonal representations. Very concretely, if $\Gamma = \Gamma_1 *_\Sigma \Gamma_2$, define the orthogonal representation $\pi : \Gamma \recht \cO(\ell^2_\R(\Gamma/\Sigma))$ given by left translation. Clearly $\pi$ is mixing relative to the subgroup $\Sigma$. One checks that there is a unique $1$-cocycle $b : \Gamma \recht H_\R$ satisfying $b(g) = 0$ for all $g \in \Gamma_1$ and $b(h) = \delta_\Sigma - \delta_{h\Sigma}$ for all $h \in \Gamma_2$. This $1$-cocycle is unbounded and vanishes on $\Sigma$.

When $\Gamma = \HNN(H,\Sigma,\theta)$ is the HNN extension generated by $H$ and $t$ subject to the relations $t \sigma t^{-1} = \theta(\sigma)$ for all $\sigma \in \Sigma$, define the orthogonal representation $\pi : \Gamma \recht \cO(\ell^2_\R(\Gamma/\Sigma))$ given by left translation. Again $\pi$ is mixing relative to $\Sigma$ and there is a unique $1$-cocycle $b : \Gamma \recht H_\R$ satisfying $b(h) = 0$ for all $h \in H$ and $b(t) = \delta_{t \Sigma}$. Also this $1$-cocycle is unbounded and vanishes on $\Sigma$.

So all groups $\Gamma$ appearing in Theorem \ref{thm.uniquecartan} belong to $\cC \cup \cD \cup \cE \cup \cC_2 \cup \cD_2$.

Use the notations as in the formulation of Theorem \ref{thm.uniquecartan}. Applying Theorem \ref{thm.main} to $A := \M_n(\C) \ot \rL^\infty(X)$, we conclude that there exists a $\Sigma \in \cS$ such that $\rL^\infty(Y) \prec A \rtimes \Sigma$ and hence $\rL^\infty(Y) \prec \rL^\infty(X) \rtimes \Sigma$.

Take any projection $q \in \D_n(\C) \ot \rL^\infty(X)$ having the same trace as $p$.
\begin{somop}
\item If $\Sigma = \{e\}$, the unitary conjugacy of $\rL^\infty(Y)$ and $(\D_n(\C) \ot \rL^\infty(X))q$ follows from \cite[Theorem A.1]{Po01}. This settles item 2 of the theorem.
\item When $\Gamma$ is a nontrivial amalgamated free product or an HNN extension and if $\Sigma$ is weakly malnormal, \cite[Proposition 8]{HPV10} provides a finite group $\Sigma_0$ such that $\rL^\infty(Y) \prec \rL^\infty(X) \rtimes \Sigma_0$. Then also $\rL^\infty(Y) \prec \rL^\infty(X)$ and the conclusion follows again from \cite[Theorem A.1]{Po01}. This settles items 3, 4 and 5 of the theorem.
\item If $\Sigma$ is relatively malnormal, take a subgroup $\Sigma < \Lambda < \Gamma$ such that $\Lambda < \Gamma$ has infinite index and $g\Sigma g^{-1} \cap \Sigma$ is finite for all $g \in \Gamma-\Lambda$. We apply Lemma \ref{lem.malnormal} below. Since the normalizer of $\rL^\infty(Y)$ is the whole of $p(A \rtimes \Gamma)p$ and since $\Lambda < \Gamma$ has infinite index, we conclude that $\rL^\infty(Y) \prec A$ and hence $\rL^\infty(Y) \prec \rL^\infty(X)$. We again find the unitary conjugacy of $\rL^\infty(Y)$ and $(\D_n(\C) \ot \rL^\infty(X))q$ from \cite[Theorem A.1]{Po01}. This settles the remaining item 1 of the theorem.
\end{somop}
So the proof of Theorem \ref{thm.uniquecartan} is complete.\hfill\qedsymbol

Our last lemma is implicitly contained in \cite[Lemma 4.2]{Va07}, but we provide an explicit proof for the convenience of the reader.

\begin{lemma}\label{lem.malnormal}
Let $\Gamma \actson (A,\tau)$ be any trace preserving action of a countable group. Assume that $\Sigma < \Lambda < \Gamma$ are subgroups such that $g \Sigma g^{-1} \cap \Sigma$ is finite for all $g \in \Gamma - \Lambda$. Put $M = A \rtimes \Gamma$. Let $p \in M$ be a projection and $B \subset pMp$ a von Neumann subalgebra. Denote by $Q$ the normalizer of $B$ inside $pMp$.

If $B \prec A \rtimes \Sigma$ and $B \not\prec A$, then $Q \prec A \rtimes \Lambda$.
\end{lemma}
\begin{proof}
Take a nonzero partial isometry $v \in \M_{1,n}(\C) \ot pM$ and a, possibly non-unital, normal $*$-homomorphism $\theta : B \recht \M_n(\C) \ot (A \rtimes \Sigma)$ satisfying $b v = v \theta(b)$ for all $b \in B$. Put $q = \theta(p)$. Write $N = \M_n(\C) \ot A$. By \cite[Remark 3.8]{Va07} we may assume that
\begin{equation}\label{eq.nonembedding}
\theta(B) \not\prec_{N \rtimes \Sigma} N \; .
\end{equation}
Whenever $\cF \subset \Gamma$, denote by $P_\cF$ the orthogonal projection onto the closed linear span of $\{a u_g \mid a \in N, g \in \cF\}$. Because of \eqref{eq.nonembedding} we can take a sequence of unitaries $b_n \in \cU(B)$ such that $\|P_\cF(\theta(b_n))\|_2 \recht 0$ for every finite subset $\cF \subset \Sigma$.

We claim that $\|E_{N \rtimes \Sigma}(x \theta(b_n) y)\|_2 \recht 0$ whenever $x,y \in (N \rtimes \Gamma) \ominus (N \rtimes \Lambda)$. Since we can approximate $x$ and $y$ by linear combinations of $a u_g$, $a \in N$ and $g \in \Gamma - \Lambda$, it suffices to prove the claim when $x = u_g$, $y = u_h$ for some $g,h \in \Gamma - \Lambda$. But then $$E_{N \rtimes \Sigma}(u_g \theta(b_n) u_h) = u_g P_{\Sigma \cap g^{-1}\Sigma h^{-1}}(\theta(b_n)) u_h \; .$$
The claim follows from the fact that $\Sigma \cap g^{-1}\Sigma h^{-1}$ is finite.

We prove that $v^* Q v \subset N \rtimes \Lambda$, so that in particular, $Q \prec A \rtimes \Lambda$. Take $d \in \cN_{pMp}(B)$. We have to prove that $v^* d v \in N \rtimes \Lambda$. Write $x = v^* d v - E_{N \rtimes \Lambda}(v^* d v)$. We have to prove that $x=0$. By construction, $x \theta(b_n) x^* = \theta(db_n d^*) xx^*$. Hence,
$$\|E_{N \rtimes \Sigma}(xx^*)\|_2 = \|\theta(db_n d^*) E_{N \rtimes \Sigma}(xx^*)\|_2 = \|E_{N \rtimes \Sigma}(x \theta(b_n) x^*)\|_2 \recht 0$$
by the claim in the previous paragraph. So, $x = 0$ and the lemma is proven.
\end{proof}

\subsection*{Proof of Theorem \ref{thm.no-group-measure-space}}

Assume that $B \subset (\rL^\infty(X) \rtimes \Gamma)^t$ is a group measure space Cartan subalgebra. Literally repeating the proof of Theorem \ref{thm.uniquecartan} it follows that $B \prec \rL^\infty(X)$. By \cite[Lemma 4.11]{OP07} the action $\Gamma \actson (X,\mu)$ is essentially free, contradicting the assumptions of Theorem \ref{thm.no-group-measure-space}.\hfill\qedsymbol

\end{document}